\documentclass{amsart}

\usepackage{kantlipsum} 
\calclayout

\usepackage{amsmath,amssymb,amscd}
\usepackage[all]{xy}
\usepackage{hyperref}
\usepackage{nccmath}

\newtheorem{theorem}{Theorem}[section]
\newtheorem{lemma}[theorem]{Lemma}
\newtheorem{corollary}[theorem]{Corollary}
\newtheorem{proposition}[theorem]{Proposition}

\theoremstyle{definition}
\newtheorem{definition}[theorem]{Definition}
\newtheorem{example}[theorem]{Example}
\newtheorem{conjecture}[theorem]{Conjecture}

\theoremstyle{remark}
\newtheorem{remark}[theorem]{Remark}
\newtheorem{caution}[theorem]{Caution}

\newtheorem{observation}[theorem]{Observation}

\numberwithin{equation}{section}

\DeclareMathOperator{\GL}{GL}

\DeclareMathOperator{\Img}{Im}

\DeclareMathOperator{\SL}{SL}

\DeclareMathOperator{\Span}{span}
\DeclareMathOperator{\T}{T}
\DeclareMathOperator{\Tr}{Tr}
\DeclareMathOperator{\Li}{Li}

\let\Theta\varTheta

\newcommand{\op}[1]{\operatorname{#1}}
\renewcommand{\t}{{\operatorname{t}}}
\newcommand{\Nm}{\mc{N}}
\newcommand{\mb}[1]{\mathbb{#1}}
\newcommand{\mc}[1]{\mathcal{#1}}
\newcommand{\mf}[1]{\mathfrak{#1}}

\renewcommand{\b}[1]{\bold{#1}}
\renewcommand{\O}{{\mc{O}}}
\newcommand{\Q}{{\mc{Q}}}
\renewcommand{\a}{{\mf{a}}}
\newcommand{\m}{{\mf{m}}}
\newcommand{\ep}{{\epsilon}}

\newcommand{\br}[1]{\overline{#1}}
\newcommand{\innerprod}[1]{\langle#1\rangle}
\newcommand{\sm}[1]{\left(\begin{smallmatrix}#1\end{smallmatrix}\right)}

\newcommand{\wtd}[1]{\widetilde{#1}}
\newcommand{\Lt}{\tilde{L}}
\renewcommand{\Re}{\operatorname{Re}}
\newcommand{\mtd}{{\wtd{m}}}
\newcommand{\modm}{\mod\!\!^{\times}\ }

\begin{document}

\title{Mahler Measure of 3D Landau-Ginzburg Potentials}
\author{Jiarui Fei}
\address{School of Mathematical Sciences, Shanghai Jiao Tong University}
\email{jiarui@sjtu.edu.cn}
\thanks{The author was supported in part by National Science Foundation of China (No. BC0710141)}


\date{}
\dedicatory{In Memory of John Horton Conway}

\begin{abstract} We express the Mahler measures of $23$ families of Laurent polynomials in terms of Eisenstein-Kronecker series.
These Laurent polynomials arise as Landau-Ginzburg potentials on Fano $3$-folds,
$16$ of which define $K3$ hypersurfaces of generic Picard rank $19$, and the rest are of generic Picard rank $< 19$.
We relate the Mahler measure at each rational singular moduli to the value at $3$ of the $L$-function of some weight-$3$ newform.
Moreover, we find $10$ exotic relations among the Mahler measures of these families.
\end{abstract}
\maketitle
\section*{Introduction}
The (logarithmic) {\em Mahler measure} of a Laurent polynomial $f\in\mb{C}[x_1^{\pm 1},x_2^{\pm 1},\dots,x_n^{\pm 1}]$ is the arithmetic average of $\log|f|$ over the $n$-dimensional torus $\mb{T}^n$:
	\begin{align*} m(f) &:= \frac{1}{(2\pi i)^n} \int_{\mb{T}^n} \log \left|f(x_1,\dots,x_n)\right| \frac{dx_1}{x_1}\cdots \frac{dx_n}{x_n}.
	\end{align*}
In the mid 1990s, Boyd \cite{Bo} (after a suggestion of Deninger) found by numerical experiment many identities of the form
$$m(f(x,y)) = s L'(E,0),$$
where the polynomial $f$ defines the elliptic curve $E$ and $s$ is some rational number.
He conjectured that such identities should hold under some additional conditions on $f$.
Many conjectural identities were verified but the general cases remain open.

The first powerful idea was introduced by F. Villegas \cite{Vi}.
Motivated by the mirror symmetry, he found appropriate families of Laurent polynomials parametrized by modular functions, then he could express their Mahler measures in terms of certain Eisenstein-Kronecker series.
Finally he linked those series to the $L$-function of elliptic curves with complex multiplication.

The second more general idea due to Deninger is related to Bloch-Beilinson's conjecture, which was also explained in \cite{Vi}.
This approach via {\em regulators} was further pursued by Mellit, Zudilin and Brunault.
Instead of referring their original papers, we recommend the new textbook \cite{BZ}.

After the elliptic curves, Bertin considered the case of $K3$ surfaces \cite{Be1,Be2}.
She treated two families, and later Samart treated four hypergeometric families \cite{Sa,Sa0}.
They together proved over $16$ identities of the similar nature.

We hope to push Villegas's ideas and Bertin's work further to other $3$-variable Laurent polynomials.
One difficulty is that a random choice of families won't have nice modular parametrization.
The first goal of this paper is to find more reasonable families to study.
Thanks to the work of Golyshev and others, we find $25$ interesting families of Laurent polynomials originated from certain version of mirror symmetry theory (see Section \ref{S:families} for details).
They were listed in Table \ref{tab:Lp} and \ref{tab:Lp2}, among which are the six families studied by Bertin and Samart.

As the first step of Villegas and Bertin's approach, we need the modular parametrization of these pencils of $K3$ hypersurfaces.
According to \cite{Go,Ga}, all the parameters $c(\tau)$ are modular functions of {\em moonshine} type, which agrees with the general phenomenon observed in \cite{LY} and proved in \cite{Dor}. 
We summarize these results in the left three columns of Table \ref{tab:forms}.
We define a complex-valued function $\mtd(c(\tau))$ whose real part agrees with the Mahler measure of $f-c(\tau)$ except in a bounded open region.
Then we are able to express the function $\mtd(c(\tau))$ in terms of certain Eisenstein-Kronecker series.
We find that some calculation made by Bertin and Samart could be generalized and performed uniformly.

\begin{theorem} \label{T:intro1} For each $f_c$ of the $25$ families except $V_2$ and $B_1$, let $\mc{F}$ be any fundamental domain of the modular function $c(\tau)$ containing $i\infty$. Then for any $\tau \in \mc{F}$ the function $\Re(\mtd(c(\tau)))$ is equal to
	$$ \frac{\Img \tau}{(2\pi)^3} \sum_{d \mid N} a_d d^2H_d(\tau),$$
	where $H_d(\tau)$ is the series 
	$\sum_{m,n}' 2\Re \left( (dm\tau +n)^{-3}(dm\br{\tau}+n)^{-1}\right) + (dm\tau +n)^{-2}(dm\br{\tau}+n)^{-2}$ and $a_d\in\mb{Z}$ is given by the column $e(\tau)$ of Table \ref{tab:forms}. 
\end{theorem}

For some appropriate choice of $\tau$, the series $H_d(\tau)$ turns out to be related to the generalized theta functions (Proposition \ref{P:dHd} and \ref{P:deH}). 
These generalized theta functions are the building blocks for certain cusp forms with rational coefficients (eg., Proposition \ref{P:CM3theta}).
Based on these, we discover and verify more identities similar to Boyd's for these Laurent polynomials.
In particular, we have the following
\begin{theorem} \label{T:intro2} For all the $25$ families except $V_2$ and $B_1$, and all known rational singular moduli of $c(\tau)$ of discriminant $D$,
	the value of $\Re(\mtd(c(\tau)))$ is equal to 
	$$\alpha L'(g_{d},0) + {\beta} L'(\chi_{d'},-1),$$
	for some newform $g_{d}$ of weight $3$ with rational coefficients, some fundamental discriminant $d'<0$, and $\alpha,\beta\in \mb{Q}$.
	Moreover, $-D/d$ is a square and $d' \mid D$.
	The complete lists are given in Section \ref{ss:list} (with $\alpha$ and $\beta$ rescaled).
\end{theorem}
\noindent The full lists contains a total of $179$ identities, which include the $7$ proved in \cite{Sa} and the $17$ conjectured in \cite{Sa0}.
The $9$ proved in \cite{Be1,Be2,BFFLM} can be derived from them as well. 
In fact, our methods could settle all the conjectured identities in \cite{Sa0}, including those involving quadratic singular moduli.
However, we do not include them because otherwise the full lists would be too long.

M. Rogers proved in \cite{Ro} two interesting relations on Mahler measures, linking the two families studied by Bertin to the hypergeometric families.
By finding the relations (modular equations) among Hauptmoduln for different genus-$0$ groups, we produce $10$ more similar relations.
\begin{theorem}	\label{T:intro3} We have 10 exotic relations on Mahler measures as listed in Theorem \ref{T:rel}.
\end{theorem}

\subsection*{Organization} Section \ref{S:families} serves as an extended introduction to explain the origin of the $25$ families of Laurent polynomials.
In Section \ref{S:MM} we recall the relation between quantum periods and Mahler measures, and review Villegas's idea in our setting.
In Section \ref{ss:Eisen} we prove our first main result -- Theorem \ref{T:mmmLG} (Mahler measures in Eisenstein-Kronecker series).
In Section \ref{S:thetaHecke}, we start by reviewing some basic material on Hecke characters for imaginary quadratic fields. Then we tailor Sch\"{u}tt's work on CM newforms with rational coefficients to our setting (Corollary \ref{C:rf3}).
In Section \ref{ss:theta} we link the relevant Hecke $L$-functions to the theta functions in Proposition \ref{P:CM3theta}.
In Section \ref{S:SM}, we conjecture a level-$N$ generalization of the Hilbert class polynomials, and explain how we reduce the search for rational singular moduli of the modular function $c(\tau)$ to reasonably small size.
In Section \ref{S:Lfun} we prove our second main result -- Theorem \ref{T:mmsv} (the $179$ identities).
In Section \ref{S:rel} we prove our last main result -- Theorem \ref{T:rel} (the $10$ exotic relations).

\subsection*{Notations and Conventions}
In this paper, $q$ is always set to be $e^{2\pi i \tau}$ for $\tau\in\mb{C}$.
The $D_x$ ($x=q,t$, etc) always denotes the differential operator $x\frac{d}{dx}$.
The bold $\b{x}$ denotes a set of variables $(x_1,x_2,\dots, x_n)$.
We write the $\eta$-quotient $\prod_i \eta(a_i\tau)^{d_i}$ in exponential form $ \prod_i a_i^{d_i}$ ($a_i, d_i\in \mb{Z}$).

\section{The 25 Families of $K3$ Hypersurfaces} \label{S:families}
\subsection{Periods of Laurent Polynomials}
Given a Laurent polynomial $f \in\mb{C}[x_1^{\pm 1},\dots,x_n^{\pm 1}]$, one can form the {quantum period} of $f$.
\begin{definition} The {\em quantum periods} of $f$ is the following integral over the $n$-dimensional torus $\mb{T}_n: |x_1|=\cdots =|x_n|=1$
	$$\pi_f(t) = \frac{1}{(2\pi i)^n} \int_{\mb{T}^n} \left(1-tf(x_1,\dots,x_n)\right)^{-1} \frac{dx_1}{x_1}\cdots \frac{dx_n}{x_n}.$$
\end{definition}
\noindent It is a (possibly multivalued) holomorphic function of $t$, and is annihilated by a {\em Picard-Fuchs operator} 
$$L_f = \sum_{i=0}^k t^i P_i(D) \in \mb{C}[t,D], \quad  D=t\frac{d}{dt}.$$

The {\em $G$-series} for a Fano variety $X$ is a generating function for certain genus-zero Gromov-Witten invariants of $X$. Since it is irrelevant to the main results of this paper, we will omit its precise definition. 
One {\em mirror conjecture} states that the Laplace transform of $G$-series for $X$ is the solution of Picard-Fuchs differential equation for some pencil of Calabi-Yau varieties that is called the {\em Landau-Ginzburg model} mirror dual to $X$.
In its most basic form, the Picard-Fuchs equation is given by the above $L_f$ for some Laurent polynomial $f$, and
the hypersurfaces defined by $f=c$ can be compactified to the required pencil of Calabi-Yau varieties.
In this case, the Laurent polynomial $f$ is called a {\em weak Landau-Ginzburg potential}.

If the Laurent polynomial $f$ is mirror to a Fano variety $X$, then one expects that $X$ can be constructed from certain smoothing of $X_f$ where $X_f$ is the toric variety associated to the Newton polytope of $f$ (Batyrev's construction \cite{Ba}).
Since a Fano variety can degenerate to many different singular toric varieties,
one might expect many Laurent polynomial mirrors for a given Fano $X$.

In \cite {ACGK} a class of Laurent polynomials called {\em Minkowski polynomials} were constructed as mirror partners to many  Fano $3$-folds.
They defined birational transformations, called {\em mutations}, that preserve periods. They showed that any two Minkowski polynomials with the same period are related by a sequence of mutations. We will give a short introduction on that in Section \ref{ss:mu}.

\begin{remark} The quantum periods can be computed by applying the residue theorem $n$ times:
	$$\pi_f(t) = \sum_{m=0}^\infty t^m \frac{1}{(2\pi i)^n} \int_{\mb{T}^n} f^m \frac{dx_1}{x_1}\cdots \frac{dx_n}{x_n}:= \sum_{m=0}^\infty b_m t^m.$$
So $b_m$ is the constant coefficient of the Laurent polynomial $f^m$.

The existence of the Picard-Fuchs operator $L_f$ is equivalent to the recurrence relation
$$\sum_{i=0}^k P_i(m-i) b_{m-i}=0 \quad \text{for any } m\in \mb{N}_0,$$
which can be guessed from first few terms of $\pi_f(t)$.
\end{remark}

\subsection{Modular Picard-Fuchs Equations}
In this subsection, we briefly recall the work of V. Golyshev and others.
	
We recall that a {\em Fano $n$-fold} $X$ is by definition a smooth $n$-dimensional complex variety with ample anticanonical divisor.
In dimension 3, according to Mori-Mukai's classification, there are exactly 105 Fano varieties up to deformation. Prior to Mukai, Iskovskikh had classified those of Picard rank $1$: there are exactly $17$ families among those $105$. The relevant invariants for the classification are 
$$\text{the index } d=[H^2(X,\mb{Z}): \mb{Z}c_1] \text{ and the level } N=\frac{1}{2d^2}\innerprod{c_1^3,[X]},$$
where $c_1$ is the anticanonical divisor. 
We will label those of index $1$ by $V_{N}$ ($N=1,\dots,9,11$), and those of index $2$ by $B_{N}$ ($N=1,\dots,5$).
The remaining two are a smooth quadratic $Q\subset \mb{P}^4$ and $\mb{P}^3$.

If $X$ is a Fano $n$-fold, then the adjunction formula implies that its anticanonical divisors are {\em Calabi-Yau}. 
In the case $n=3$, this will be a family of $K3$ surfaces of Picard rank $20-\rho$, where $\rho$ is the Picard rank of $X$, which can range from $1$ to $6$.
We recall that for a general $K3$ surface $Y$ the second homology group $H_2(Y,\mb{Z})$ is free of rank $22$.
As a free subgroup, $\op{Pic}(Y)$ can have rank ranging from $1$ to $20$.
If $\rho(Y)=20$ then we say that the $K3$-surface is {\em singular}.
If a one-parameter family $Y_t$ of $K3$ surfaces with generic Picard rank $\rho$, then one expect that the associated Picard-Fuchs equation has order $22-\rho$.
This is mostly due to the fact that $\int_{\gamma}\omega  =0$ for any $\gamma \in \op{Pic}(Y_t)$ (see \cite[Proposition 5.2]{VY} for details).
Here, $\omega$ is the unique (up to scalar multiplication) holomorphic differential 2-form on $Y_t$.

Golyshev constructed in \cite{Go} a specific collection of $17$ pencils of $K3$ surfaces mirror to the $17$ smooth Fano $3$-folds of Picard rank $1$.
In particular, he described the corresponding Picard-Fuchs equations and their modular properties (see Section \ref{ss:MMM}). 
He found that all of them are of type $D3$, which is a specific class of determinantal linear differential equations of order $3$.
They can be written as the following form
\begin{align}\label{eq:D3} D^3 + t(D+\frac{1}{2})(\alpha_3(D^2+D)+\beta_1) +t^2(D+1)(\alpha_2(D+1)^2+\beta_0) &+ \alpha_1t^3(D+2)(D+\frac{3}{2})(D+1) \notag \\
&+\alpha_0 t^4(D + 3)(D + 2)(D + 1). \end{align}
The corresponding Laurent polynomials for the above $17$ families were given by V. Przyjalkowski in \cite[Table 1]{Pr}. 
We slightly modify them by mutations and shifts to our preferred form and list them together with their Picard-Fuchs equations in the upper part of Table \ref{tab:Lp} and Table \ref{tab:Lp2}.
\footnote{The Laurent polynomials of $V_{24}$ and $B_{6b}$ are what Bertin labelled by $Q$ and $P$ \cite{Be1,Be2}, and those of $V_2, V_4, V_6, V_8$ are equivalent to what Samart labelled by $D,C,B,A$ in \cite{Sa,Sa0}. The recurrence relation of the period sequence of $V_{12}$ is the same as the one used in Ap\'{e}ry's proof of the irrationality of $\zeta(3)$ \cite{BP}.}

However, S. Galkin found $8$ more Fano $3$-folds with Picard rank $>1$ satisfying $D3$ equations, and verified their modular properties in \cite{Ga}.
Y. Prokhorov and he observed that for some complex structure they admit a finite group action $G\curvearrowright X$ such that $\op{Pic}^G(X) = \mb{Z} c_1(X)$.
We list their mirror Laurent polynomials in the lower part of Table \ref{tab:Lp} and Table \ref{tab:Lp2}.
We record their Picard ranks here:
$$\rho(V_{12a})=2,\ \rho(V_{12b})=3,\ \rho(V_{20})=2,\ \rho(V_{24})=4,\ \rho(V_{28})=2,\ \rho(V_{30})=3,\ \rho(B_{6a})=2,\ \rho(B_{6b})=3.$$

The corresponding results for the remaining $80$ Fano $3$-folds of the Mori-Mukai classification were conjectured by T. Coates, A. Corti, S. Galkin and A. Kasprzyk together with Golyshev and were proved in all cases as an application of their Fanosearch Program \cite{CCGK}. Their work gives new explicit descriptions of the Fano varieties as well as a Laurent polynomial defining the mirror family in every case.
However, all their Picard-Fuchs equations have order greater than $3$.

%
\renewcommand{\arraystretch}{1.5} 
\begin{table}
		\caption{\label{tab:Lp} List of Laurent polynomials and Picard-Fuchs equations for $d=1$}
	$\begin{array}{|c|c|c|} \hline
	\text{label}	& f & -(\alpha_3,\alpha_2,\alpha_1;\beta_1,\beta_0)\\ \hline 
	V_2	& (x_1+x_2+x_3+ 1)^6/(x_1x_2x_3)  & 48(36,0,0;5,0)  \\ \hline
	V_4 & (x_1+x_2+x_3+ 1)^4/(x_1x_2x_3) & 16(16,0,0;3,0) \\ \hline 
	V_6	& (x_1+1)^2(x_2+x_3+ 1)^3/(x_1x_2x_3) & 12(9,0,0;2,0)\\ \hline
	V_8	& (x_1 + 1)^2(x_2 + 1)^2(x_3 + 1)^2/(x_1x_2x_3)  & 16(4,0,0;1,0) \\ \hline
	V_{10}& (x_1 + 1)(x_2 + 1)^2(x_3 + 1)(x_1 + x_3 + 1)/(x_1x_2x_3) & 4(11,4,0;3,-1) \\ \hline
	V_{12}& (x_1 + 1)(x_3 + 1)(x_2 + x_1x_2 + 1)(x_3 + x_2x_3 + 1)/(x_1x_2x_3) & (34,-1,0; 10,0)\\ \hline
	V_{14}& (x_1 + x_1x_2 + 1)(x_2 + x_2x_3 + 1)(x_3 + x_1x_3 + 1)/(x_1x_2x_3) & (26,27,3; 8,0) \\ \hline
	V_{16}&  (x_1 + 1)(x_2 + x_3 + x_2x_3)(x_1 + x_3 + x_1x_2 + 1)/(x_1x_2x_3) & (24,-16,0; 8,0) \\ \hline	
	V_{18}&  (x_1 + 1)(x_3 + 1)(x_2 + x_2x_3 + x_1x_2^2x_3 + x_1x_2x_3 + 1)/(x_1x_2x_3) & (18,27,0; 6,0)\\ \hline
	V_{22}& (x_1 + x_2 + x_1x_2 + x_2x_3)(x_1 + x_3 + x_1x_3 + x_2x_3)/(x_1x_2x_3)+1/(x_2x_3)  &  (20,-56,44; 8,-8)\\ \hline \hline
	V_{12a}	& (x_1 + x_2)(x_1 + 1)(x_2 + 1)(x_3 + 1)^2/(x_1x_2x_3) & (28,128,0; 8,-32) \\ \hline 
	V_{12b}	& (x_1 + x_1x_2 + 1)(x_2 + x_1x_2 + 1)(x_3 + 1)^2/(x_1x_2x_3)  & (40,0,144;12,0) \\ \hline 
	V_{20} 	& (x_1x_2x_3 + 1)(x_1^{-1} + 1)(x_2^{-1} + 1)(x_3^{-1} + 1)  &  (12,-144,0; 4,36) \\ \hline 
	V_{24}	& (x_1 + x_2 + x_3 + 1)(x_1^{-1} + x_2^{-1} + x_3^{-1} + 1) & (20,-64,0; 8,0) \\ \hline 
	V_{28}	& x_1 + 2x_2 + x_3 + (x_1x_2)^{-1} + x_2x_3 + x_1x_2^{-1} + x_2x_3^{-1} + 2x_2^{-1} + 1 &  (6,47,28;2,4) \\ \hline 
	V_{30}	& x_1 + x_2 + x_3 + x_1^{-1} + x_2^{-1} + x_3^{-1} + x_1x_2^{-1} + x_2x_3^{-1} + x_3x_1^{-1} + 3 & (14,-29,60; 6,-4) \\ \hline 
	\end{array}$	
	
	{\raggedright \hspace{1.2cm} The right column lists the coefficients of the $D3$ equation \eqref{eq:D3}. All of them have $\alpha_0=0$. \par}
\end{table}


\subsection{Mutation Invariance} \label{ss:mu}
Readers can safely skip this subsection to reach the main results of this article.
Consider a Laurent polynomial $f=\sum_{i=k}^l C_i(x_1,x_2) x_3^i$ ($k<0$ and $l>0$).
A monomial change of variables 
$$(x_1,x_2,x_3) \mapsto (x_1^{a_{11}}x_2^{a_{12}}x_3^{a_{13}},x_1^{a_{21}}x_2^{a_{22}}x_3^{a_{23}},x_1^{a_{31}}x_2^{a_{32}}x_3^{a_{33}} )$$
is called a $\GL_3(\mb{Z})$-equivalence if the integral matrix $(a_{ij})$ is invertible.

Suppose that each $C_i$ is a Laurent polynomial in $x_1$ and $x_2$ such that $A(x_1,x_2)^i C_i(x_1,x_2)$ remains Laurent.
Then the pullback of $f$ along the birational transformation $(\mb{C}^\times)^3 \dashrightarrow (\mb{C}^\times)^3$ given by
\begin{equation}\label{eq:mu} (x_1,x_2,x_3) \mapsto (x_1,x_2,A(x_1,x_2)x_3) \end{equation}
is another Laurent polynomial
$$g = \sum_{i=k}^l A(x_1,x_2)^i C_i(x_1,x_2)x_3^i.$$

\begin{definition}[{\cite{ACGK}}] A {\em mutation} is a birational transformation $(\mb{C}^\times)^3 \dashrightarrow (\mb{C}^\times)^3$ given by a composition of:
	\begin{enumerate}
		\item a $\GL_3(\mb{Z})$-equivalence;
		\item a birational transformation of the form \eqref{eq:mu};
		\item another $\GL_3(\mb{Z})$-equivalence.
	\end{enumerate}
	If $f$ and $g$ are Laurent polynomials and $\varphi$ is a mutation such that $\varphi^* f =g$ 
	then we say that $f$ and $g$ are related by the mutation $\varphi$.
\end{definition}

\begin{lemma}[{\cite{ACGK}}] \label{L:muinv} If $f$ and $g$ are related by a mutation, then the quantum periods of $f$ and $g$ coincide.
\end{lemma}

\begin{example} Consider the following Laurent polynomials \begin{align*}
	f_0&=(1 + x_1)(1 + x_2)(1 + x_1 + x_3)(1 + x_2x_3 +x_3)/(x_1x_2x_3),\\
	f_1&=(1 + x_1)(1 + x_2)(1 + x_3)(1 + x_1 + x_2 + x_3 + x_1x_2)/(x_1x_2x_3), \\
	f_2&=(1 + x_1)(1 + x_2)(1 + x_1 + x_3)(1 + x_1 + x_2 + x_3)/(x_1x_2x_3).
	\end{align*}
	Then $f_0$ can be mutated to $f_1$ by the transformation 
	$$(x_1,x_2,x_3) \mapsto \Big(x_1,x_2,\frac{1+x_1}{x_3} \Big)$$
	and $f_1$ can be mutated to $f_2$ by the $\GL_3(\mb{Z})$-equivalence $(x_1,x_2,x_3) \mapsto (x_1,x_3^{-1},x_2)$
	followed by the transformation 
	$$(x_1,x_2,x_3) \mapsto \Big(x_1,x_2,\frac{1+x_1}{x_3} \Big).$$	
	According to the appendix of \cite{ACGK}, the mutation-equivalent class of $f_0$ contains at least $71$ $\GL_3(\mb{Z})$-nonequivalent Laurent polynomials supporting on reflexive polytopes.
\end{example}

\section{Mahler Measure} \label{S:MM}
\subsection{Mahler Measure and Periods}  \label{ss:MMP}

Let $f$ be a Laurent polynomial in $n$ variables. 

\begin{definition} The (logarithmic) {\em Mahler measure} of $f$ is the arithmetic average of $\log|f|$ over the $n$-dimensional torus $\mb{T}^n$:
	\begin{align*} m(f) &:= \frac{1}{(2\pi i)^n} \int_{\mb{T}^n} \log \left|f(x_1,\dots,x_n)\right| \frac{dx_1}{x_1}\cdots \frac{dx_n}{x_n}\\
	& = \int_0^1\cdots \int_0^1 \log \left|f(e^{2\pi i\theta_1},\dots,e^{2\pi i\theta_n}) \right| d\theta_1\cdots d\theta_n.
	\end{align*}
\end{definition}
\noindent It is a nontrivial fact \cite[Proposition 3.1]{BZ} that the integral defining $m(f)$ always converges.

From now on we will denote by $\{f_c\}$ the family of Laurent polynomials $f-c$ for $c\in\mb{C}$,
and by $f_c$ the particular member $f-c$.
One of the ideas in \cite{Vi} is to study $m(f_c)$ as a function of the complex parameter $c$.
Sometimes it is more convenient to work with the parameter $t=1/c$.

Let $\mc{K}$ be compact region given by the image of the torus $\mb{T}^n$ under the map $\b{x} \mapsto f(\b{x})$.
Then $1-t f$ does not vanish on $\mb{T}^n$ for $t^{-1} \notin \mc{K}$.
For $t^{-1} \in \mb{C} \setminus \mc{K}$, we define the holomorphic function
\begin{align} \mtd(t) &:= -\log(t) + \frac{1}{(2\pi i)^n} \int_{\mb{T}^n} \log (1- tf) \frac{dx_1}{x_1}\cdots \frac{dx_n}{x_n} \notag\\
& = -\log(t) - \sum_{n=1}^\infty \frac{b_n}{n} t^n.  \label{eq:mb} 
\end{align}  
Here and throughout we take the principal branch of the logarithm.

\begin{lemma}[{\cite{Vi}}] \label{L:mtd} We have that $m(f_t)=\Re(\wtd{m}(t))$ for $1/t=c \in \mb{C} \setminus \mc{K}^\circ$, where $\mc{K}^\circ$ is the interior of $\mc{K}$.
\end{lemma}

\begin{caution} In general,  $m(f_t)$ may not agree with $\Re(\wtd{m}(t))$ in the interior of $\mc{K}$.
\cite[Example 1]{Vi} contains such an example.
\end{caution}

\begin{remark} \label{r:regionK}
In view of the above caution, it worthwhile to give a complete description of the region $\mc{K}$ for our $25$ Laurent polynomials.
This seems not an easy task. However, it is elementary to find the (real) maximum and minimum of $\mc{K}$.
The maximums are particular easy.
Since all Laurent polynomials have positive coefficients, the maximum on $\mb{T}^n$ are equal to $f(1,1,1)$.
We also observe that the polynomials for $V_N$ ($N=8,12_a,12_b,20,24$), $B_4$ and $B_{6b}$ are reciprocal, i.e.,
$$f(x_1,x_2,x_3) = f(x_1^{-1},x_2^{-1},x_3^{-1}).$$
So these polynomials only take real values on $\mb{T}^n$. In these cases we have that $\mc{K}^\circ = \emptyset$.
Finally, we remark that exactly the same proof as Lemma \ref{L:muinv} shows that the Mahler measure is also mutation-invariant.
\end{remark}

\subsection{The Mirror-Moonshine for the 25 Families}  \label{ss:MMM}
We have listed the Picard-Fuchs equations satisfied by the quantum periods of the $25$ Laurent polynomials. 
All of them have the {\em maximal unipotent monodromy} at zero. 
In terms of the Frobenius method, this is equivalent to say that apart from the quantum period $u_0(t)=1+\sum_{n=1}^\infty b_nt^n$ as a holomorphic solution around $t=0$,
there is a second solution $u_1(t)$ and a third solution $u_2(t)$ of the form 
\begin{align*} u_1(t) &= u_0(t)\log(t) + v_1(t),\\
u_2(t) &= \frac{1}{2} u_0(t)\log^2(t) + v_1(t)\log(t) + v_2(t),
\end{align*}
where $v_1(t)=\sum_{n\geq 1} b_{1,n}t^n$ and $v_2(t)=\sum_{n\geq 2} b_{2,n} t^n$ are holomorphic around $0$.
In this paper, $u_2(t)$ is irrelevant to our discussion.

Following the similar argument as in \cite{Vi},
we define
$$\tau = \frac{1}{2\pi i} \frac{u_1}{u_0},$$
then
$$q = e^{2\pi i\tau} = t + \cdots.$$
So we can locally invert $q$ around $0$ and obtain the so-called {\em mirror map}
$$t(\tau) = q + \cdots.$$
Let \begin{align*}
u_0(\tau) = u_0(t(\tau)) = 1 + \sum_{n\in \mb{N}} c_nq^n, \\
e(\tau) = u_0(\tau) \frac{D_q t(\tau)}{t(\tau)}  = 1 + \sum_{n\in \mb{N}} e_nq^n, 
\end{align*}
where $D_q = q\frac{d}{dq} = \frac{1}{2\pi i} \frac{d}{d\tau}$ is the usual differential for modular forms.

Finally we notice that the equality \eqref{eq:mb} can be rewritten as 
$$\mtd(t) = -\log t - \int_{0}^t (u_0(s)-1) \frac{ds}{s}.$$
By change of variables $s=t(\tau)$ we obtain an expression for $\mtd$ as a function of $\tau$. Note that $t=0$ corresponds to $\tau=i\infty$.
\begin{theorem}[\cite{Vi}] \label{T:mmm} Locally around $\tau = i\infty$, we have
	\begin{equation}\label{eq:local}  \mtd(\tau) = -2\pi i \tau - \sum_{n=1}^\infty \frac{e_n}{n} q^n, \quad q=e^{2\pi i\tau}.  \end{equation}
\end{theorem}

It is a nontrivial fact that $t(\tau)$ is a modular function (i.e., a meromorphic modular form of weight $0$).
Before Golyshev, this had been proved by C. Doran in \cite{Dor} in the context of lattice-polarized $K_3$ surfaces.
It then follows from the general theory of modular forms that $u_0(\tau)$ and $e(\tau)$ are meromorphic modular forms of weight $2$ and $4$ respectively.
It turns out that $u_0(\tau)$ and $e(\tau)$ are genuine modular forms except for $V_2$ and $B_1$. In fact, they are all Eisenstein series (see Table \ref{tab:forms}). 

Alternatively one can start with a modular form $u_0(\tau)$ and a modular function $t(\tau)$ in the same row of Table \ref{tab:forms}, and employ the following general fact.
If $f(\tau)$ is an arbitrary modular form of positive weight $k$ and $t(\tau)$ a modular function, 
then the power series $F(t)$ obtained by expressing $f(\tau)$ locally as a power series in $t(\tau)$ always satisfies a linear differential equation 
of order $k + 1$ with algebraic (if $t(\tau)$ is a Hauptmodul, even polynomial) coefficients. 
A discussion of this phenomenon in general, and an algorithm to find the explicit linear differential equation, can be found in \cite[5.4]{Zabook} and \cite{Yi}.

It worths mentioning that the modular functions $c(\tau)=t(\tau)^{-1}$ are all of {\em moonshine} type \cite{CN}. 
Lian and Yau first observed this phenomenon in \cite{LY} and formulate their mirror-moonshine conjecture, 
which roughly says that for any pencil of $K3$ surfaces of generic Picard rank $19$, (the reciprocal of) the mirror map $c(\tau)=t(\tau)^{-1}$ is {\em commensurable} with some {\em McKay-Thompson series}.
The conjecture was proved by Doran in \cite{Dor}, and Galkin verified in \cite{Ga} that the mirror-moonshine also holds for the $8$ families in Table \ref{tab:Lp2} (of Picard rank $<19$).
In fact, for our $25$ families the function $c(\tau)$ is itself a McKay-Thompson series.

Let $\Gamma_0(N)$ be the congruence subgroup of $\SL_2(\mb{Z})$
$$\Gamma_0(N):=\{\sm{a&b\\c&d} \in\SL_2(\mb{Z}) \mid c\equiv 0 \mod N\}.$$
We recall the definition of the {\em Atkin-Lehner involution} $W_n$ for any $n\mid N$ such that $\gcd(n,N/n)=1$.
Note that the number of such divisors of $N$ is $2^{\sigma_0^+(N)}$, where $\sigma_0^+(N)$ is the number of distinct prime factors of $N$.
Over $\mb{C}$ they may be defined as elements of $\SL_2(\mb{R})$ as follows.
Let $a,b,c,d\in \mb{Z}$ be such that $adn-bcN/n=1$ and define $W_n = \frac{1}{\sqrt{n}}\sm{na & b\\ Nc & nd}$.
This construction is well-defined up to (left and right) multiplication by $\Gamma_0(N)$.
Moreover, modulo $\Gamma_0(N)$ we have the following relations
$$W_{n_1}W_{n_2} = W_{n_1n_2 / \gcd(n_1,n_2)^2}.$$
Let $W(N)$ be the group generated by all $W_{n}$.
Following the notations of \cite{CN}, we denote by $\Gamma_0^{+n}(N)$ the subgroup of $\SL_2(\mb{R})$ generated by $\Gamma_0(N)$ and $W_n$,
and by $\Gamma_0^+(N)$ the subgroup generated by $\Gamma_0(N)$ and $W(N)$.
Then $\Gamma_0(N)$ is a normal subgroup of $\Gamma_0^+(N)$ and the quotient $\Gamma_0^+(N)/\Gamma_0(N)$ is isomorphic to $\mb{Z}_2^{\sigma_0^+(N)}$.
We see from Table \ref{tab:forms} that the genus-$0$ congruence subgroups of $\SL_2(\mb{R})$ associated to those Thompson series of level $N$
are all equal to $\Gamma_0^+(N)$ except for $V_{12}$.

\renewcommand{\arraystretch}{1.5}
\begin{table}
	\caption{\label{tab:forms} List of Forms} 
	$\begin{array}{|c|c|c|c|c|} \hline
	\text{label} & \text{group} & c(\tau)=t(\tau)^{-1} & e(\tau) & u_0(\tau) \\ \hline  
	V_2	   &\Gamma_0(1) & j=(h^8+(2h^{-2})^8)^3,\ h=\frac{1^{1}}{2^{1}}&  E_4^{1/2} & E_6E_4^{-1/2} \\ \hline
	V_4    &\Gamma_0^+(2) & (h^{12}+64h^{-12})^2,\ h=\frac{1^{1}}{2^{1}} & 80(-1,1) & 24(1,-1)  \\ \hline
	V_6	   &\Gamma_0^+(3) & (h^{6}+27h^{-6})^2,\ h=\frac{1^{1}}{3^{1}} & 30(-1,1) & 12(1,-1)\\ \hline
	V_8	   &\Gamma_0^+(4) & h^{24},\ h=\frac{2^2}{1^14^1} & 16(-1,0,1) & 8(1,0,-1) \\ \hline
	V_{10}  &\Gamma_0^+(5) & h^6+22+125h^{-6},\ h=\frac{1^1}{5^1}& 10(-1,1) & 6(1,-1)  \\ \hline
	V_{12}  &\Gamma_0^{+6}(6) & h^{12},\ h=\frac{2^13^1}{1^1 6^1} & (-7,1,-1,7) & (5,-1,1,-5) \\ \hline
	V_{14}   &\Gamma_0^+(7) & (h^{2}+7h^{-2})^2,\ h=\frac{1^1}{7^1} & 5(-1,1) & 4(1,-1) \\ \hline
	V_{16}   &\Gamma_0^+(8) & h^{8},\ h=\frac{2^14^1}{1^1 8^1} & (-4,1,-1,4) & 2(2,-1,1,-2)  \\ \hline
	V_{18}   &\Gamma_0^+(9) & h^{6},\ h=\frac{3^2}{1^1 9^1} & 3(-1,0,1) & 3(1,0,-1) \\ \hline
	V_{22}   &\Gamma_0^+(11) & (1+3h)^2(1+3h+h^{-1}),\ h=\frac{3^1 33^1}{1^1 11^1} & 2(-1,1) & 12(1,-1)+\theta_{11} \\ \hline \hline 
	V_{12a}  & \Gamma_0^+(6) & (h^{3}+8h^{-3})^2,\ h=\frac{1^1 3^1}{2^1 6^1} &6(-1,-1,1,1) & 4(1,1,-1,-1) \\ \hline
	V_{12b}  & \Gamma_0^+(6) & (h^{2}+9h^{-2})^2,\ h=\frac{1^1 2^1}{3^1 6^1} &8(-1,1,-1,1) & 6(1,-1,1,-1) \\ \hline
	V_{20}  &  \Gamma_0^+(10) & (h^{2}+4h^{-2})^2,\ h=\frac{1^1 5^1}{2^1 10^1}  &2(-1,-1,1,1) & 2(1,1,-1,-1) \\ \hline 
	V_{24}	  & \Gamma_0^{+}(12)   & h^6,\ h=\frac{2^26^2}{1^13^14^112^1} &(-2,1,2,-2,-1,2) & 4(1,-1,-1,1,1,-1) \\ \hline
	V_{28}  & \Gamma_0^+(14) & (h^{2}-h^{-2})^2,\ h=\frac{2^1 7^1}{1^1 14^1} &(-1,-1,1,1) & 4(1,1,-1,-1)+\theta_{14} \\ \hline 
	V_{30}  & \Gamma_0^+(15) & (h+3h^{-1})^2,\ h=\frac{1^1 5^1}{3^1 15^1}  &(-1,-1,1,1) & 3(1,1,-1,-1)+\theta_{15} \\ \hline   
	\end{array}$
	
	{\raggedright \vspace{.1cm} The modular forms $e(\tau)$ and $u_0(\tau)$ are given by the coefficients $a_d$ and $a_d'$ in \eqref{eq:eq} and \eqref{eq:uq}, and $\theta_N$ is some cusp form of weight $2$ and level $N$. \par}
\end{table}

\subsection{Expressing Mahler Measure in Eisenstein-Kronecker Series} \label{ss:Eisen}
Let $E_2$ and $E_4$ be the Eisenstein series of weight $2$ and $4$ on $\Gamma_0(1)$:
\begin{align*}
E_2(\tau) &= 1 + 24 \sum_{n=1}^\infty \sigma_1(n)q^n, \\
E_4(\tau) &= 1 + 240 \sum_{n=1}^\infty \sigma_3(n)q^n.
\end{align*}
We also set 
\begin{align*}
G_{2,d} (\tau)= \sum_{n \geq 1} \sigma_1(n) q^{dn},\\
G_{4,d} (\tau)= \sum_{n \geq 1} \sigma_3(n) q^{dn}.
\end{align*}


By explicit calculation, we find for our 25 families except $V_2$ and $B_1$ that
all $u_0(\tau)$ are (up to a shift)
\footnote{Unlike the form $e(\tau)$, the modular form $u_0(\tau)$ is not invariant under a shift of the function value of $f$.
Such a shift of $f$ will give rise to a shift of $u_0(\tau)$ by some weight-$2$ cusp form (of the same level).} 
Eisenstein series of weight $2$ of the form:
\begin{equation} \label{eq:uq} 1+\sum_{d \mid N} a_d' dG_{2,d}  = \sum_{d \mid N} \frac{a_d'}{24}dE_2(d\tau) \qquad (a_d' \in\mb{Z}),
\end{equation}
and all $e(\tau)$ are Eisenstein series of weight $4$ of the form
\begin{equation} \label{eq:eq} 1-\sum_{d \mid N} a_d d^2G_{4,d}  = -\sum_{d \mid N} \frac{a_d}{240}d^2E_4(d\tau) \qquad (a_d \in\mb{Z}).
\end{equation} 
Note that the equations imply that \begin{align}  \label{eq:240}
\sum_{d \mid N} a_d' d &= 24,\\
\sum_{d \mid N} a_d d^2 &= 240,
\end{align}
which is equivalent to say that both series are of level $N$.
In fact, we can further observe that $a_d = -a_e$ and $a_d' = -a_e'$ if $de=N$. At this moment, we do not need this observation.

Now according to Lemma \ref{L:mtd} and Theorem \ref{T:mmm}, to compute the Mahler measure of $f_{c(\tau)}$, we need to compute
\begin{align} \Re\Big( -2\pi i \tau - \sum_{n\in\mb{N}} \frac{e_n}{n}q^n \Big) &= \Re\Big( -2\pi i \tau + \sum_{n\in \mb{N}} \frac{\sigma_3(n)}{n} \sum_{d \mid N} a_d dq^{dn}  \Big) \notag \\
&= \Re\Big( -2\pi i \tau  + \sum_{d \mid N} a_d \sum_{n\in \mb{N}} {\sigma_3(n)} d \frac{q^{dn}}{n}  \Big)   \notag \\
&= \Re\Big( -2\pi i \tau + \sum_{d \mid N} a_d \sum_{n\in \mb{N}} d^{-1} D_q^2 \left( \Li_3(q^{dn}) \right) \Big).  \label{eq:Li} \end{align}
The last equality due to the relation 
$$D_q^2(\Li_3(q^{dn})) = (dn)^2 \Li_1(q^{dn})\ \text{ for } n,d\in\mb{N}.$$

For this, following \cite{Be1,Sa} we introduce 
$$F_d(\xi)=\sum_{n\in\mb{N}} \Li_3(q^{nd}+\xi).$$
It is known that the Fourier series of $F_d(\xi)$ converges to $F_d(\xi)$ at $\xi =0$, i.e.,
$F_d(0) = \sum_{n\in\mb{Z}} \hat{F}_d(n)$, where the Fourier coefficients
$$\hat{F}_d(n) = \begin{cases} -\frac{1}{2\pi i} \sum_{m\geq 1} (m^3(dm\tau -\frac{n}{4}))^{-1} & \text{if } 4\mid n, \\ 0 & \text{otherwise.} \end{cases}$$

Since $D_q^2 = -\frac{1}{4\pi^2}\frac{d^2}{d\tau^2}$, we have that
\begin{equation} \label{eq:DFd} d^{-1}D_q^2(F_d(0)) = \frac{1}{4\pi^3i} \sum_{n\in\mb{Z}, m \geq 1} \frac{d}{m(dm\tau-n)^3}.
\end{equation}
Then we can continue our computation 
\begin{align} \eqref{eq:Li} &= \Re \Big(-2\pi i \tau + \sum_{d \mid N} a_d d^{-1} D_q^2\big( F_d(0)\big) \Big) \notag \\
&=\Re \bigg(-2\pi i \tau + \frac{1}{4\pi^3i} \sum_{d \mid N} a_d  \sum_{n\in\mb{Z}, m \geq 1} \frac{d}{m(dm\tau-n)^3} \bigg) \notag \\   
&=\Img \bigg(2\pi \tau +  \frac{1}{8\pi^3} \sum_{d\mid N} a_d \sum_{n\in\mb{Z}, m \neq 0} \frac{1}{m}  \frac{d}{(dm\tau-n)^3}  \bigg) \label{eq:inter}
\end{align}

It is straightforward to verify that
\begin{equation} \label{eq:Im} \frac{1}{m} \Img \frac{1}{(m\tau +n)^3} = -\Img \tau \left( 2\Re\bigg(\frac{1}{(m\tau+n)^3(m\br{\tau}+n)}\bigg) + \frac{1}{(m\tau+n)^2(m\br{\tau}+n)^2} \right).
\end{equation}
Moreover, \begin{align} \label{eq:zeta4}
\sum_{d\mid N} a_d \sum_{n\neq 0, m=0} 2\Re\bigg(\frac{d^2}{(dm\tau+n)^3(m\br{\tau}+n)}\bigg) + \frac{d^2}{(dm\tau+n)^2(m\br{\tau}+n)^2}  = \sum_{d\mid N} a_d d^2  \sum_{n\in \mb{Z}}  \frac{3}{n^4} = 3\cdot 240 \cdot  \frac{\pi^4}{45}.
\end{align}

Combining \eqref{eq:Im} and \eqref{eq:zeta4} with \eqref{eq:inter}, we find that \eqref{eq:inter} is equal to
\begin{equation}
\frac{\Img \tau}{(2\pi)^3} \sum_{d \mid N} a_d d^2 \Big(\sum_{m,n}' 2\Re \left( (dm\tau +n)^{-3}(dm\br{\tau}+n)^{-1}\right) + (dm\tau +n)^{-2}(dm\br{\tau}+n)^{-2} \Big),
\end{equation}
where the summation $\sum_{m,n}'$ means $\sum_{(m,n)\in \mb{Z}^2\setminus\{0,0\}}$.

We set \begin{equation} \label{eq:Hd}
H_d(\tau) = \sum_{m,n}' 2\Re \left( (dm\tau +n)^{-3}(dm\br{\tau}+n)^{-1}\right) + (dm\tau +n)^{-2}(dm\br{\tau}+n)^{-2},
\end{equation}
and thus obtain the following theorem.
\begin{theorem} \label{T:mmmLG} For each $f_c$ of the $25$ families except $V_2$ and $B_1$, let $\mc{F}$ be any fundamental domain of the modular function $c(\tau)$ containing $i\infty$. Then for any $\tau \in \mc{F}$ the function $\Re(\mtd(c(\tau)))$ is equal to
	$$ \frac{\Img \tau}{(2\pi)^3} \sum_{d \mid N} a_d d^2H_d(\tau),$$
where $H_d(\tau)$ is defined as above and $a_d$ is given by the column $e(\tau)$ of Table \ref{tab:forms}.
\end{theorem}

\begin{caution}The equality in Theorem \ref{T:mmm} only holds locally around $i\infty$.
So it is necessary to choose a fundamental domain containing $i\infty$.
As a simple example, let us consider the family $V_4$. In this case, $c(\tau)=\big(\frac{\eta(\tau)^{12}}{\eta(2\tau)^{12}}+64\frac{\eta(2\tau)^{12}}{\eta(\tau)^{12}}\big)^2$ is a Hauptmodul for the monodromy group $\Gamma_0^+(2)$.
One can check that $c\big(\frac{1}{\sqrt{-2}}\big) = c\big(\frac{1}{3}(\frac{1}{\sqrt{-2}} -1)\big)=256$ but the $\mtd\big(\frac{1}{\sqrt{-2}}\big) \neq \mtd\big(\frac{1}{3}(\frac{1}{\sqrt{-2}} -1)\big)$.
The correct value for the Mahler measure is the former because $\frac{1}{3}\big(\frac{1}{\sqrt{-2}} -1\big)$ is in a fundamental domain away from $i\infty$.
\end{caution}

The result for families $V_{24}$ and $B_{6b}$ were proved in \cite{Be1}, and for families $V_4,V_6$ and $V_8$ were proved in \cite{Sa}.
It will turn out that for certain choices of $\tau$, $H_d(\tau)$ is a sum of the special values of two partial Hecke $L$-functions, which are both related to theta functions  (see Proposition \ref{P:dHd} and \ref{P:deH}).

\section{Generalized Theta Functions and Hecke $L$-functions} \label{S:thetaHecke}
\subsection{Imaginary Quadratic Fields and Quadratic Forms}
We review some elementary facts about orders in imaginary quadratic fields and their relation to binary quadratic forms, mostly taken from \cite{Co}.
Let $K$ be a quadratic field, then $K=\mb{Q}(\sqrt{n})$ for a unique squarefree integer. 
Recall that the discriminant $d_K$ of $K$ is defined to be 
$$d_K = \begin{cases} n & \text{if } n\equiv 1 \mod 4\\ 4n & \text{otherwise.} \end{cases}$$
An {\em order} $\O$ in $K$ is a subring $\O\subset K$ which is a finitely generated $\mb{Z}$-module of rank $[K:\mb{Q}]$.
The ring of integer $\O_K$ of $K$ is the maximal order of $K$.
Since $\O$ and $\O_K$ are free $\mb{Z}$-modules of rank $2$, it follows that $m=[\O_K:\O]<\infty$.
$m$ is called the {\em conductor} of $\O$, and $D=m^2 d_K$ is by definition the {\em discriminant} of $\O$.
Then we have that 
$$\O = \mb{Z} + m\O_K = \Span_{\mb{Z}}\left(1, m\frac{d_K+\sqrt{d_K}}{2}\right).$$

\begin{definition} Given an ideal $\m$ of $O_K$, an $\O$-ideal $\a$ is {\em prime to} $\m$ if $\a+\m\O=\O$.
\end{definition}
\noindent In most situations we consider, $\m$ will just be the principal ideal $(m)$ or $(1)$.
Being prime to $(m)$ is equivalent to that $\gcd(\Nm(\a),m)=1$ (\cite[Lemma 7.18]{Co}), where $\Nm(\a)=|\O/\a|$ is the {\em ideal norm} of $\a$.
Let
\begin{align*}
I(\O,\m)&:= \text{group of invertible fractional $\O$-ideals prime to $\m$}, \\
P(\O,\m)&:=  \text{subgroup of $I(\O,\m)$ generated by principal ideals.}
\end{align*}
When $\m=(m)$, we write $I(\O,m)$ and $P(\O,m)$ for $I(\O,(m))$ and $P(\O,(m))$. 
When $\m=(1)$, we write $I(\O)$ and $P(\O)$ for $I(\O,(1))$ and $P(\O,(1))$. 
We call the quotient 
$$Cl(\O) := I(\O)/P(\O)$$
the {\em ideal class group} of $\O$.
When $\O$ is the maximal order $\O_K$, $I(\O_K)$ and $P(\O_K)$ will be denoted by $I_K$ and $P_K$.
The assignment $\a \mapsto \a\O_K$ gives an isomorphism $I(\O,m) \cong I_K(m)$ with the inverse given by $\a \mapsto \a\cap \O$.
\begin{lemma}[{\cite[Proposition 7.22]{Co}}] \label{L:ClO} Let $K$ be an imaginary quadratic field, and let $\O$ be an order of conductor $m$ in $K$. 
	Then there are natural isomorphisms 
	\begin{equation} \label{eq:isoClO} Cl(\O) \cong I(\O,m)/P(\O,m) \cong I_K(m) / P_{K,\mb{Z}}(m),   \end{equation}
where $P_{K,\mb{Z}}(m) := \{\alpha\O_k \mid \alpha\in \O_K \text{ and } \alpha \equiv a \mod m\O_K \text{ for some integer $a$ relatively prime to $m$}\}.$
\end{lemma}

We will denote a binary quadratic form $f(x,y) = ax^2+bxy+cy^2$ by the triple $[a,b,c]$.
A quadratic form $[a,b,c]$ is called {\em primitive} if $\gcd(a,b,c)=1$.
We denote by $\Q$ (resp. $\Q^0$) the set of all (resp. primitive) positive definite quadratic forms ({\em p.d.q.f.}), 
and by $\Q_D$ (resp. $\Q_D^0$) the subset of $\Q$ (resp. $\Q^0$) of all quadratic forms of {\em discriminant} $D=b^2-4ac$.
Consider the equivalence relation given by the natural $\SL_2(\mb{Z})$-action on $\Q$:
$$f_1 \sim f_2 \iff f_1(x,y) = f_2(px+qy,rx+sy) \text{ for some $\sm{p&q\\r&s}\in \SL_2(\mb{Z})$}. $$
Note that the subset $\Q^0$ and $\Q_D$ are $\SL_2(\mb{Z})$-invariant.

We denote by $Cl(D)$ the equivalence classes of $\Q_D^0$. 
Gauss and Dirichlet defined a composition on $Cl(D)$ making it an abelian group \cite{Co}.
We mention that the identity element is the class containing the form \begin{align}
\left[1,0,-D/4\right] & &\text{if $D\equiv 0 \mod 4$},\\
\left[1,1,(1-D)/4 \right] & &\text{if $D\equiv 1 \mod 4$},
\end{align}
and the inverse of the class containing  $[a,b,c]$ in $Cl(D)$ is the class containing $[a,-b,c]$.

\begin{theorem}[{\cite[Theorem 7.7]{Co}}] \label{T:formclass} Let $\O$ be the order of discriminant $D$ in $K$, and $[a,b,c]\in \Q_D^0$. Then the map  
\begin{equation}\label{eq:bijection} \mc{I}:  [a,b,c] \mapsto \Span_{\mb{Z}}\bigg( a,\frac{b-\sqrt{D}}{2} \bigg)  
\end{equation}
induces an isomorphism $Cl(D) \cong Cl(\O)$.
\end{theorem}
\noindent Hereafter we will freely identify $Cl(\O)$ with $Cl(D)$. 
Since $[a,b,c]\sim [c,-b,a]$ in $Cl(D)$, we also have a twin map $\check{\mc{I}}: [a,b,c] \mapsto \Span_{\mb{Z}}\Big(c, \frac{b+\sqrt{D}}{2} \Big)$ inducing the same isomorphism.

\subsection{CM Newforms with Rational Coefficients}
In this subsection, we tailor Sch\"{u}tt's work \cite{Sc} to our setting.
For unexplained (standard) terminology about modular forms in this section, we refer readers to the textbook \cite{Zabook, Iw}.

Let $\m$ be an ideal of an imaginary quadratic field $K$ and $\ell\in\mb{N}$. 
We denote \begin{align*}
P_{K,1^\times}({\m}) :&= \{\alpha\O_K \mid \alpha\in K^\times,\  \alpha\equiv 1 \modm \m \},
\end{align*}
where $\!\modm$ is the {\em multiplicative congruence}. This is a subgroup of $P_K(\m)$.

\begin{definition} A {\em Hecke character} $\psi$ of $K$ modulo $\m$ with $\infty$-type $\ell$ is a homomorphism
$$\psi: I_K(\m) \to \mb{C}^\times $$	
such that for all $\alpha \in K_{1^\times,{\m}}$, we have 
	$$\psi(\alpha \O_K) = \alpha^\ell.$$
The ideal $\m$ is called the {\em conductor} of $\psi$ if it is minimal in the sense that if $\psi$ is defined modulo $\m'$, then $\m | \m'$.

The definition requires some elaboration. We denote by $X_K^\ell(\m)$ the set of all Hecke characters of $K$ modulo $\m$ with $\infty$-type $\ell$.
To define an element in $X_K^\ell(\m)$ completely, one need to further specify a character 
$$\phi: I_K(\m)/P_{K,\mb{Z}}(\m) \to \mb{C}^\times$$
and a character 
$$\eta_{\mb{Z}}: P_{K,\mb{Z}}(\m) / P_{K,1^\times}(\m) \to \mb{C}^\times.$$
Recall that for $\m=(m)$, $I_K(m)/P_{K,\mb{Z}}(m)$ is isomorphic to the class group $Cl(\O)$.
The following lemma is well-known (eg., \cite[Proposition 1.2]{Sh}).

\begin{lemma} $P_{K}(\m) / P_{K,1^\times}(\m) \cong (\O_K/\m)^\times$ and the subgroup $P_{K,\mb{Z}}(\m) / P_{K,1^\times}(\m)$ of $P_{K}(\m) / P_{K,1^\times}(\m)$ can be identified with 
	$$G_{\mb{Z}}:=\{\alpha \in (\O_K/\m)^\times \mid \alpha \equiv a \mod \m,\ \exists a\in \mb{Z}   \}.$$
\end{lemma}
\noindent Then we can observe that $\eta_{\mb{Z}}$ can be recover from $\psi$ by
$$a\mapsto \frac{\psi(a\O_K)}{a^\ell}\quad (a\in\mb{Z},\ \gcd(a,\Nm(\m))=1).$$ 

\begin{definition}
	For any Hecke character $\psi\in X_K^\ell(\m)$ we define $f_{\psi}(\tau)=\sum_{n\in \mb{N}} a_nq^n$ by
	$$f_{\psi}(\tau) := \sum_{\a\text{ integral}} \psi(\a)q^{\Nm(\a)} .$$
\end{definition}
\noindent In this definition we tacitly extended $\psi$ by $0$ for all fractional ideals of $K$ which are not prime to $\m$.

Let $\chi_K$ be the quadratic character of conductor $|d_K|$. 
\begin{theorem}[Hecke, Shimura] $f_{\psi}$ is a Hecke eigenform of weight $\ell+1$, level $N=\mc{N}(\m) |d_K|$ and nebentypus character $\varepsilon=\chi_K\eta_{\mb{Z}}$:
	$$f_{\psi} \in \mc{S}_{\ell+1}(\Gamma_0(N), \chi_K\eta_{\mb{Z}}).$$
	Moreover, $f_{\psi}$ is a newform if and only if $\m$ is the conductor of $\psi$.	
\end{theorem}

We say $\psi$ or $f_{\psi}$ has rational (or real) coefficients if all Fourier coefficients $a_n$ are rational (or real).
Such condition will impose very strong restriction on the characters $\phi, \eta_{\mb{Z}}$ and the (imaginary quadratic) field $K$.
But let us first mention (\cite[Corollary 1.2]{Sc}) that if a newform $f$ has real coefficients and its weight $k$ is odd, then $f$ has {\em complex multiplication} by its nebentypus $\varepsilon$, that is,
$$f= f\otimes \varepsilon = \sum_{n\in\mb{N}} \varepsilon(n)a_nq^n.$$ 

Suppose that $f_{\psi}$ has real coefficients, then the character $\eta_{\mb{Z}}$ is automatically determined. 
More precisely, if $\ell$ is odd, then $\eta_{\mb{Z}}=\chi_K$ so $\varepsilon$ is trivial;
if $\ell$ is even, then $\eta_{\mb{Z}}=1$ and thus $\varepsilon = \chi_K$ (\cite[Corollary 1.5]{Sc}).
So the only freedom left is the choice of the character $\phi: Cl(\O) \to \mb{C}^\times$ if $\m = (m)$. 
If we further assume that $f_{\psi}$ has rational coefficients, then \cite[Proposition 3.1]{Sc} says that $e_K \mid \ell$ where $e_K$ is the exponent of the class group $Cl(\O_K)$. We remark that in general $e_{\O} \nmid \ell$ where $e_{\O}$ is the exponent of $Cl(\O)$.

Conversely, suppose that $e_K \mid \ell$ and $\a\in I_K(m)$ has order $n$ in $Cl(\O_K)$ with $\a^{n} = \alpha \O_K$. 
For each character $\phi: Cl(\O) \to \mb{C}^\times$, we define a Hecke character $\Phi \in X_K^\ell(m)$ by
\begin{equation} \label{eq:Phi} \Phi(\a) = \phi([\a \cap \O]) {\alpha}^{\ell/n}. \end{equation}
In this case $f_{\Phi}$ is also denoted by $f_{\phi,\ell}$.
We recall an elementary but useful criterion.
\begin{lemma}[{\cite[Lemma 2.2]{Sc}}] Let $\psi$ be a Hecke character of an imaginary quadratic field $K$. Then $f_{\psi}$ has rational coefficients if and only if $\Img \psi \subset \O_K$.
\end{lemma}
\noindent In particular, if the image of $\phi$ is contained in $\O_K^\times$, then we see from \eqref{eq:Phi} that $f_{\Phi}$ has rational coefficients.
For our application, we summarize the discussion so far in the next corollary. The ``moreover" part is due to the construction in \cite[Lemma 4.1 and Proposition 5.1]{Sc}.
\begin{corollary}\label{C:rf3} Let $K$ be an imaginary quadratic field and $\ell$ is a multiple of the exponent $e_K$ of $Cl(\O_K)$.
Let $D=m^2 d_K$ for some $m\in \mb{N}$. Then any character $\phi: Cl(D)\to \O_K^\times$ yields a Hecke character $\Phi \in X_{K}^\ell(m)$ such that $f_{\Phi}=f_{\phi,\ell}$ has rational coefficients.
	Moreover, if $\ell$ is even, any $f_{\psi}$ of weight $\ell+1$ with rational coefficients arises this way.
\end{corollary}


\begin{caution} \label{c:newform} (1). In general, the $f_{\phi,\ell}$ in Corollary \ref{C:rf3} may not be a newform. See Proposition \ref{P:tonew} for examples.
	(2). For the case $\ell=2$ that we mainly concern, the fact that $e_K \mid \ell$ restricts us to imaginary quadratic fields whose class group consists only of $2$-torsion. 
	However, higher torsion may appear in the group $Cl(\O)\cong Cl(D)$. 
	But it follows from Remark \ref{r:smalllist} that only $\mb{Z}_2^g$ $(g=0,1,2)$ can appear as $Cl(D)$ in our application if Conjecture \ref{conj:HDN} holds.
\end{caution}

The modular forms are related to $L$-functions via the {\em Mellin transform}.
The {\em Hecke $L$-function} of $\psi$ is the Mellin transform of $f_{\psi}$:
	$$L(\psi,s) = \sum_{\a\text{ integral}} \frac{\psi(\a)}{\Nm(\a)^s}.$$
\end{definition}
\noindent To any ideal $\a$ of $\O$, we associate a {\em partial} Hecke series
$$Z(\ell,\a,s) = \frac{1}{w} \sum_{\lambda\in \a}' \frac{\lambda^\ell}{(\lambda \br{\lambda})^s},$$
where $w$ is the number of units in $\O$. Recall that $w$ is equal to $6$ or $4$ for $D=-3$ or $-4$, and to $2$ otherwise.

The following lemma is a straightforward variation of the well-known statement for maximal orders.
\begin{lemma} \label{L:sumpartial} Let $\O$ be an order of $K$ of conductor $m$, and $\psi\in X_K^\ell(m)$. Then
	$$L(\psi,s) = \sum_{[\a] \in Cl(\O)} \frac{\Nm(\a\O_K)^s}{\psi(\a\O_K)}  Z(\ell,\a,s).$$
\end{lemma}
\begin{proof}  Recall the isomorphism \eqref{eq:isoClO}. We shall consider the partition of $I_K(m)$ by the classes in $Cl(\O)$.
	For each $[\a_0] \in Cl(\O)$, we choose some $\a\in [\a_0]^{-1}$ so that $\a\a' = \lambda \O$ for any $\a'\in [\a_0]$.
	Then $\lambda\in \a$ and we have that 
	\begin{align*}\sum_{\a'\in [\a_0]} \frac{\psi(\a'\O_K)}{\Nm(\a' \O_K)^s} =  \sum_{\lambda\in (\a\setminus \{0\})/\O^\times} \frac{\psi(\a^{-1}\O_K) \psi(\lambda)}{\Nm(\a\O_K)^{-s} \Nm(\lambda)^s}   =   \frac{\Nm(\a\O_K)^s}{\psi(\a\O_K)}  \frac{1}{w} \sum_{\lambda\in \a}' \frac{\psi(\lambda)}{\Nm(\lambda)^s}.
	\end{align*}
	So \begin{align*}L(\psi,s) = \sum_{[\a_0] \in Cl(\O)}\sum_{\a'\in [\a_0]} \frac{\psi(\a'\O_K)}{\Nm(\a' \O_K)^s}  = \sum_{[\a] \in Cl(\O)}  \frac{\Nm(\a\O_K)^s}{\psi(\a\O_K)}  Z(\ell,\a,s).
	\end{align*}
\end{proof}

\begin{remark} \label{r:extend} This implies that each summand $\frac{\Nm(\a\O_K)^s}{\psi(\a\O_K)}  Z(\ell,\a,s)$ does not depend on the choice of the representative $\a$ in $Cl(\O)$.
	By naturally extending the definition of $\psi$, we may even allow $\a$ not prime to $m$.
	Indeed,	if $\a$ is not prime to $m$, by the first isomorphism in Lemma \ref{L:ClO} we can always choose an ideal $\mf{b}$ prime to $m$ representing the same element in $Cl(\O)$ as $\a$. 
	Then $\mf{b} = \a (x)$ for some $x$, so we define $\psi(\a\O_K) = \psi(\mf{b}\O_K)x^{-\ell}$. We have that
	\begin{align*}	\frac{\Nm(\mf{b}\O_K)^{s}}{\psi(\mf{b}\O_K)} Z(\ell,\mf{b},s) &= \frac{\Nm(\a\O_K)^{s}\Nm(x)^s}{\psi(\a\O_K)\psi(x)} Z(\ell,\a(x),s) \\
	&= \frac{\Nm(\a\O_K)^{s}}{\psi(\a\O_K)} \frac{\Nm(x)^s}{x^\ell} \sum_{\lambda\in\mf{a}} \frac{x^\ell}{(x\br{x})^s} \frac{\lambda^\ell}{(\lambda\br{\lambda})^s}  \\
	&= \frac{\Nm(\mf{a}\O_K)^{s}}{\psi(\mf{a}\O_K)} Z(\ell,\mf{a},s).
	\end{align*}
\end{remark}

\subsection{Generalized Theta Functions} \label{ss:theta}
Let $A=(a_{ij})$ be a symmetric positive definite matrix of rank $r$. The {\em Laplace operator} attached to $A$ is
$$\Delta_A = \sum_{i,j} a_{ij}^* \frac{\partial}{\partial x_i \partial x_j}, \quad \text{where } (a_{ij}^*)=A^{-1}.$$
A {\em spherical function} for $A$ is a homogeneous polynomial $P$ in $r$ variables such that 
$$\Delta_A P = 0.$$

Now we assume $A$ is integral and {\em even}, which means the diagonal entries of $A$ are even. So $n^{\t} A n$ is even for any $n\in\mb{Z}^r$.
Let $N$ be such that $NA^{-1}$ has the same properties.
\begin{definition} The generalized {\em theta function} attached to $A$ and $P$ is defined as
	$$\Theta_{A,P} (\tau) = \sum_{n\in \mb{Z}^r} P(n) \exp\Big(\frac{1}{2}n^{\t} A n \tau\Big).$$
\end{definition}
\noindent Clearly, $\Theta_{A,P}$ is linear w.r.t. $P$.

\begin{theorem}[{\cite[Theorem 10.9]{Iw}}] Let $P$ be a spherical function for $A$ of degree $v$. Then
	$$\Theta_{A,P}(\tau) \in \mc{M}_{v+r/2} (\Gamma_0(N),\chi_D),$$ 
where $D=(-1)^{r/2} |A|$ and $\chi_D$ is the Kronecker character. 
Moreover, if $v>0$, then $\Theta_{A,P}(\tau)$ is a cusp form.
\end{theorem}

\begin{example} A positive definite quadratic form $Q=[a,b,c]$ corresponds to an even symmetric positive definite matrix $\sm{2a&b\\b&2c}$. Then 
	$$\theta_Q(\tau) := \Theta_{Q,1}(\tau) = \sum_{m,n\in\mb{Z}} q^{am^2 + b mn +cn^2}$$ 
	is the ordinary $\theta$-series attached to $Q$.
\end{example}

\begin{example} For two positive definite quadratic forms $Q=[a,b,c]$ and $P=[a',b',c']$ satisfying 
	$$\innerprod{Q,P}:=2ca'-bb'+2ac'=0,$$
 we define the theta function
	 $$\Theta_{Q,P}(\tau) = \sum_{m,n\in\mb{Z}} \frac{1}{2} (a' m^2 + b' mn +c' n^2) q^{am^2 + b mn +cn^2}.$$  
	It is a cusp form of weight $3$.	
For any quadratic form $Q=[a,b,c]$, we denote
$$\Theta_Q(\tau) := \sum_{m,n\in\mb{Z}} \frac{1}{2}(a m^2  - cn^2) q^{am^2 + b mn +cn^2}.$$
Note that $\Theta_{gQ,gP} = \Theta_{Q,P}$ for $g\in \SL_2(\mb{Z})$ but in general $\Theta_{gQ}\neq \Theta_{Q}$.
\end{example}

\begin{definition} The linear span of $\Theta_{Q,P}$ for all quadratic forms $P$ spherical for $Q$ is denoted by $\innerprod{\Theta_Q}$.
	The {\em theta kernel} of $Q$ consists of all $P$ such that $\Theta_{Q,P}=0$.
\end{definition}

\begin{lemma}\label{L:dim=1} If $Q$ is a p.d.q.f. of order at most $2$, then the dimension of $\innerprod{\Theta_{Q}}$ is at most $1$.
\end{lemma}
\begin{proof} It is clear that the dimension of $\innerprod{\Theta_{Q}}$ is at most $2$. 
If $Q=[a,b,c]$ is a form of order at most $2$, then $g\cdot [a,b,c]=[a,-b,c]$ for some $g\in \SL_2(\mb{Z})$, so there is some nontrivial $g'=\sm{1&0\\0&-1}g \in\GL_2(\mb{Z})$ such that $g'Q=Q$.
Since $Q\notin Q^\perp$ and $g'$ cannot fix the entire $\Q$, $g'$ cannot fix $Q^\perp$. 	
So there is some $P\in Q^\perp$ such that $P'=g' P \neq P$.
Then we have that 
$$\Theta_{Q,P} = \Theta_{g'Q,g'P} = \Theta_{Q,P'}.$$
Hence $\Theta_{Q,P-P'}=0$, and therefore the dimension of $\innerprod{\Theta_{Q}}$ is at most $1$.
\end{proof}

\begin{lemma} \label{L:ka} For $k\in \mb{Z}$ and $\gcd(a,c)=1$, the form $Q=[a,ka,c]$ has order at most $2$, and its $\Theta$-kernel contains $[0,2,k]$.
Moreover, let $\O$ be an order of discriminant $D=k^2a^2-4ac$ in $K$, then we have that $\mc{I}([a,ka,c])^2 = a\O$
where $\mc{I}$ is the map in \eqref{eq:bijection}.
\end{lemma}
\begin{proof} We check by straightforward calculation that 
	$$g\cdot [a,ka,c] = [a,-ka,c]\ \text{ for }\ g=\sm{1&-k\\0&1}.$$ So the form has order at most $2$.
We run the simple algorithm in the proof of Lemma \ref{L:dim=1}, and find that $g'=\sm{1&-k\\0&-1}$ and $P=[a,0,-c]\in Q^\perp$ is not fixed by $g'$.
Then $P-g'P = ak[0,2,k]$ lies in the $\Theta$-kernel.

Recall that $\mc{I}([a,ka,c])=\Span_{\mb{Z}}(a,r)$ where $r=\frac{ka-\sqrt{D}}{2}$.
The norm of the ideal $\mc{I}([a,ak,c])$ can be computed as the GCD of $(\Nm(a),\Tr(a\br{r}),\Nm(r))$ which is $(a^2,-a^2k/2,ac)=a$.
Suppose that $\mc{I}([a,ka,c])^2=(u)$, then $\Nm(u) = a^2$.
Using $\gcd(a,c)=1$, we check that $a$ is the only element in $\mc{I}([a,ka,c])$ with norm $a^2$. Hence $u=a$.
\end{proof}

As before, $K$ is an imaginary quadratic field, and $\O$ is an order in $K$ of discriminant $D=m^2d_K$.
\begin{lemma} \label{L:partialHL} Suppose $[\a] \in Cl(\O)$ is represented by a form $Q=[a,b,c]$ such that $\a^2=a \O$ and $[0,2a,b]$ is in the $\Theta$-kernel of $Q$. 
	Then for $\psi\in X_{K}^2(m)$ determined by $\phi: Cl(\O)\to \mb{C}^\times$ as \eqref{eq:Phi}, we have $$\frac{\Nm(\a\O_K)^s}{\psi(\a\O_K)}  Z(2,\a,s) = \phi([\a])L(\Theta_Q,s).$$
\end{lemma}
\begin{proof} By Remark \ref{r:extend} and Theorem \ref{T:formclass}, we may assume that $\a = \Span_{\mb{Z}}\big(a, \frac{b-\sqrt{D}}{2}\big)$. Then
	 \begin{align*} 
	 Z(2,\a,s) &= \frac{1}{2} \sum_{\lambda\in \a}' \frac{\lambda^2}{(\lambda \br{\lambda})^s} = \frac{1}{2} \sum_{m,n\in\mb{Z}}'\frac{a(am^2-cn^2) + \frac{1}{2}(b-\sqrt{D})(2amn+bn^2)}{a^s(am^2+bmn+cn^2)^s}. 
	 \end{align*} 
By our assumption that $[0,2a,b]$ is in the $\Theta$-kernel, the last one is equal to
$$\sum_{m,n\in\mb{Z}}' \frac{a(am^2-cn^2)}{2a^s(am^2+bmn+cn^2)^s} = a^{1-s}L(\Theta_Q,s). $$
Then by the assumption that $\a^2=a\O$, we have
$$\frac{\Nm(\a\O_K)^s}{\psi(\a\O_K)}  = \frac{a^{s}}{\phi([\a]) a}.$$	
Hence our desired equality follows.
\end{proof}

\begin{remark}\label{r:kc} (1). The similar proof can show that for any $\a$, $Z(\ell,\a,s)=L(\Theta_{Q,P},s)$ where $P$ is some spherical function of degree $\ell$ for $Q$.
However, only the above special case appear in our application.

(2). There are similar statement as Lemma \ref{L:ka} for the form $Q=[a,kc,c]$. Its $\Theta$-kernel contains $[k,2,0]$, and we have that $\check{\mc{I}}([a,kc,c])^2 = (c)$, where $\check{\mc{I}}$ is the twin map of $\mc{I}$.
Moreover, the conclusion of Lemma \ref{L:partialHL} also holds for $\a$ satisfying $\a^2=c\O$ and $[b,2c,0]$ is in the $\Theta$-kernel of $Q$.
\end{remark}


\begin{proposition} \label{P:CM3theta} Suppose that the form class group $Cl(D)$ is isomorphic to $\mb{Z}_2^g$ where $D=m^2d_K$ and $g\in\mb{N}_0$. 
	Let $Q_i=[a_i,b_i,c_i]\ (i=1,2,\cdots, 2^g)$ be a set of representatives in $Cl(D)$ such that 
\begin{align*} & \text{the $\Theta$-kernel of $Q_i$ is spanned by $[0,2a_i,b_i]$ and $\mc{I}(Q_i)^2=(a_i)$.}
\end{align*}
Then the Hecke $L$-series for $\Phi \in X_K^2(m)$ given by \eqref{eq:Phi} is equal to
	$$L(\Phi,s)=\sum_{i=1}^{2^g} \phi(Q_i) L(\Theta_{Q_i},s).$$
Moreover, any cusp eigenform of weight $3$ and level $D$ with rational coefficients is of the form $\sum_i \phi(Q_i) \Theta_{Q_i}$.
\end{proposition}
\begin{proof} Let $Q=[a,b,c]$ be one of those $Q_i$'s, and let $\a = \mc{I}(Q)$. 
Then the equality follows from Lemma \ref{L:sumpartial} and Lemma \ref{L:partialHL}.
The last statement is obtained from the inverse Mellin transform and Corollary \ref{C:rf3}.
\end{proof}

\noindent {\bf Convention}: For cusp forms (of a fixed odd weight) induced from the characters $Cl(D) \to \O_K^\times$, we shall label them by the level $N=m^2|d_K|$ and a letter as follows. We first normalize all forms such that the Fourier coefficient of $q$ is $1$. 
For two cusp forms $g_1$ and $g_2$ at the same level, we say $g_1<g_2$ if $g_1$ is a newform but $g_2$ is not, or both are newforms (or oldforms) and the Fourier coefficient sequences of $g_1$ is less than that of $g_2$ in the lexicographical ordering.
Then we assign the letters to those forms at the same level according to this ordering.
This labelling is compatible with the one in the LMFDB \cite{LMFDB} in the sense that
if $g_{N x}$ is a newform then $g_{N x}$ is also the LMFDB label if we ignore the first letter (for the Dirichlet character).
If $g_{N x}$ is an oldform, we will add a circle superscript $g_{N x}^\circ$ to indicate it is an oldform.

\begin{example} \label{ex:g84} There are 4 form classes for $D=-84$ represented by 
	$$Q_1=(1,14,70),\ Q_2=(2,14,35),\ Q_3=(7,14,10),\ Q_4=(14,14,5),$$
	which corresponds to $(0,0),(1,0),(0,1),(1,1)$ in $\mb{Z}_2 \times \mb{Z}_2$.
By Lemma \ref{L:ka} the four forms all satisfy the conditions in Proposition \ref{P:CM3theta}.
So all weight-$3$ newforms of level $84$ with rational coefficients are
\begin{align*}
g_{84a} = \Theta_{Q_1}-\Theta_{Q_2}+\Theta_{Q_3}-\Theta_{Q_4}=q-2q^2-3q^3+\cdots, \\
g_{84b} = \Theta_{Q_1}-\Theta_{Q_2}-\Theta_{Q_3}+\Theta_{Q_4}=q-2q^2+3q^3+\cdots, \\
g_{84c} = \Theta_{Q_1}+\Theta_{Q_2}+\Theta_{Q_3}+\Theta_{Q_4}=q+2q^2-3q^3+\cdots, \\
g_{84d} = \Theta_{Q_1}+\Theta_{Q_2}-\Theta_{Q_3}-\Theta_{Q_4}=q+2q^2+3q^3+\cdots. \\
\end{align*}
\end{example}

As we remarked in \ref{c:newform}, the construction does not necessarily produce newforms.
We make a short list for all such cases appearing in our application. All the identities can be easily verified by Sturm's theorem.
\begin{proposition} \label{P:tonew} We have the following equalities:
\begin{align*} g_{28b}^\circ(\tau) &= g_7(\tau) +3 g_7(2\tau) +8 g_7(4\tau), \\
g_{32b}^\circ(\tau) &= g_8(\tau) + 2g_8(2\tau) + 8 g_8(4\tau),  \\
g_{44a}^\circ(\tau) &= g_{11}(\tau) + 8g_{11}(4\tau), \\
g_{48b}^\circ(\tau) &= g_{12}(\tau) + 8 g_{12}(4\tau),  \\
g_{60a}^\circ(\tau) &= g_{15b}(\tau) - g_{15b}(2\tau) + 8 g_{15b}(4\tau), \\ 
g_{60b}^\circ(\tau) &= g_{15a}(\tau) + g_{15a}(2\tau) + 8 g_{15a}(4\tau), \\ 
g_{72b}^\circ(\tau) &= g_{8}(\tau) + 2g_{8}(3\tau) + 27g_{8}(9\tau), \\ 
g_{96c}^\circ(\tau) &= g_{24b}(\tau) -2 g_{24b}(2\tau) + 8g_{24b}(4\tau),  \\
g_{99b}^\circ(\tau) &= g_{11}(\tau)+ 5 g_{11}(3\tau) + 27 g_{11}(9\tau), \\ 
g_{112b}^\circ(\tau) &= g_{7}(\tau) + 3g_{7}(2\tau) + 12g_{7}(4\tau) + 24g_{7}(8\tau) + 64g_{7}(16\tau), \\
g_{180d}^\circ(\tau) &= g_{20b}(\tau) + 4g_{20b}(3\tau) + 27g_{20b}(9\tau),  \\
g_{192c}^\circ(\tau) &= g_{12}(\tau) + 8g_{12}(4\tau) + 64g_{12}(16\tau),  \\
g_{240c}^\circ(\tau) &= g_{15b}(\tau) - g_{15b}(2\tau) + 12 g_{15b}(4\tau) -8g_{15b}(8\tau) + 64g_{15b}(16\tau), \\
g_{240d}^\circ(\tau) &= g_{15a}(\tau) + g_{15a}(2\tau) + 12 g_{15a}(4\tau) +8g_{15a}(8\tau) + 64g_{15a}(16\tau).  
\end{align*}	
\end{proposition}
\noindent Then one can use the obvious relation $L(g(n\tau),s) = n^{-s} L(g(\tau),s)$ to get relations between the corresponding $L$-functions.

\subsection*{Lattice Sums} \label{ss:LS}
When $\ell=0$, the analogue of Proposition \ref{P:CM3theta} reads:
\begin{observation} \label{P:CM1theta} Let $Q_i=[a_i,b_i,c_i]$ be a set of representatives in $Cl(D)$ with $D=m^2d_K$.
	Then the Hecke $L$-series for $\Phi \in X_K^0(m)$ given by \eqref{eq:Phi} is equal to
	$$L(\Phi,s)=\sum_{i} \phi(Q_i) L(\theta_{Q_i},s).$$
\end{observation}
\noindent Recall that $\theta_{Q}$ is the ordinary theta series attached to $Q$.
The associated $L$-function $L(\theta_{Q},s)$ is called the {\em Epstein zeta-function}, which is also denoted by $\zeta_Q(s)$.

As before, let $\chi_d$ be the Kronecker character. For two discriminants $d_1$ and $d_2$ with product $D$,
we define \cite{GKZ} $$\chi_{d_1,d_2}(Q) = \chi_{d_1}(p)= \chi_{d_2}(p),$$
where $p$ is any prime represented by $Q$ and not dividing $D$.
The definition is independent of the choice of $p$.

If $Cl(D)$ has at most $2$-torsion, then any $\phi$ is a quadratic character (the so-called {\em genus character}).
It is known that $\phi = \chi_{d_1,d_2}$ for two discriminants $d_1,d_2$ with $d_1d_2=D$. 
In general, $d_1$ and $d_2$ need not be {\em fundamental} (i.e., a discriminant of a quadratic field), and $(d_1,d_2)$ is not unique.
Finding a particular pair $(d_1,d_2)$ is easy and completely elementary.

\begin{example}\label{ex:genusch} Recall the notation from Example \ref{ex:g84}.
For the genus character $\phi$ sending $(Q_1,Q_2,Q_3,Q_4)$ to $(1,1,-1,-1)$, we have that $\phi = \chi_{14,-7}$
	because \begin{align*}
\chi_{14,-7}(Q_1) &= \chi_{-7}(Q_1(3,-1))=\chi_{-7}(37)=1, \\
\chi_{14,-7}(Q_2) &= \chi_{-7}(Q_2(3,-1))= \chi_{-7}(11)=1, \\
\chi_{14,-7}(Q_3) &= \chi_{-7}(Q_3(-1,2))=\chi_{-7}(19)=-1, \\
\chi_{14,-7}(Q_4) &= \chi_{-7}(Q_4(0,0))= \chi_{-7}(5)=-1.
	\end{align*}
Similarly, the genus characters given by $(1,-1,-1,1)$, $(1,-1,1,-1)$, and $(1,1,1,1)$ are equal to $\chi_{21,-4}$, $\chi_{28,-3}$, and $\chi_{1,-84}$ respectively.
\end{example}

\begin{theorem}[Dirichlet-Kronecker] \label{T:genusch} Suppose that $D=m^2d_K=d_1d_2$ for two fundamental discriminants $d_1$ and $d_2$.
	Let $\phi=\chi_{d_1,d_2}$ be the associated genus character on $Cl(D)$ and $\Phi\in X_K^0(m)$. 
	Then we have the factorization
	$$L(\Phi,s) = \omega(D) L(\chi_{d_1},s)L(\chi_{d_2}, s),$$
where $w(D)=6$ if $D=-3$, $w(D)=4$ if $D=-4$, otherwise $w(D)=2$.
\end{theorem}
\noindent Theorem \ref{T:genusch} was classically stated for genus characters of $Cl(\O_K)$ (eg. \cite[Theorem 12.7]{Iw}).
The proof for this more general version is similar. We write both sides as Euler products, then compare each $p$-factor.

\section{Rational Singular Moduli} \label{S:SM}
\subsection{}
We first briefly review the classical class polynomial for the $j$-invariant.
For any primitive p.d.q.f. $Q$, we write $\tau_Q$ for the (unique) root of the $Q(x,1)$ in the upper half plane $\mb{H}$.
The value of a modular function $\vartheta$ at an irrational $\tau_Q$ is called a {\em singular modulus} of $\vartheta$. They are algebraic integers.
Recall the {\em classical (Hilbert) class polynomial} for the $j$-invariant is 
$$H_{D}(X) = \prod_{Q\in \Q_{D}^0/\Gamma_0} \big(X-j(\tau_Q) \big).$$
\begin{theorem}[{\cite{Zabook}}]\label{T:classpoly} The polynomial $H_D(X)$ belongs to $\mb{Z}[X]$ and is irreducible.
\end{theorem}
\noindent In particular, the rational singular moduli of $j$ are precisely those $j(\tau_Q)$ where $Q$ is the unique orbit in $\Q_D^0/\Gamma_0$.

Our next goal is to calculate the Mahler measure of $f_c$ for $c=c(\tau)$ a rational singular modulus. 
So we hope to find all $Q$ such that $c(\tau_Q)$ is rational and $\tau_Q$ lie in the fundamental domain of $c(\tau)$ containing $i\infty$.  

\subsection{} Throughout $a,b,c$ are always integers.
Following \cite{GKZ}, we denote 
\begin{align*} \Q_{D,N}^0 &= \{[aN,b,c]\in \Q_{D} \mid \gcd(a,b,c)=1 \},\\
\Q_{D,N,m}^0 &= \{[aN,b,c]\in \Q_{D,N}^0 \mid \gcd(N,b,ac)=m \}.
\end{align*}
The set $\Q_{D,N}^0$ is clearly stable under the action of $\Gamma_0(N)$. Let $\mc{M}_{D,N}$ be the set of positive integers $m$ such that $\Q_{D,N,m}^0$ is non-empty. Recall the group $W(N)$ generated by the Atkin-Lehner involutions $W_n$.
Here are three simple observations.
\begin{observation} (1). For $m\in \mc{M}_{D,N}$, we have that $m \mid N$, $m^2 \mid D$ and $D/m^2$ is a discriminant.
$m$ is a product of two coprime numbers $m_1=\gcd(N,b,a)$ and $m_2=\gcd(N,b,c)$.\\
(2). If $N$ is square-free, then $\mc{M}_{D,N}$ contains at most one element.\\
(3). The group $W(N)$ acts on $\Q_{D,N}^0/\Gamma_0(N)$ and the action restricts to $\Q_{D,N,m}^0/\Gamma_0(N)$.
\end{observation}

Fix a solution $\beta\mod 2N$ of $\beta^2 \equiv D\mod 4N$. 
Define the subset of $\Q_{D,N}^0$
$$\Q_{D,N,\beta}^0 = \{Q=[a,b,c]\in \Q_{D,N}^0 \mid b\equiv \beta \mod 2N\}.$$
The set $\Q_{D,N,\beta}^0$ is also stable under the action of $\Gamma_0(N)$, but not stable under $W(N)$.
We have that \cite[(7)]{GKZ}
$$W_{n}: \Q_{D,N,\beta}^0/\Gamma_0(N) \cong \Q_{D,N,\beta^*}^0/\Gamma_0(N),\qquad \beta^*\equiv \begin{cases} \beta \mod 2N/n \\ -\beta \mod 2n\end{cases}.$$

Fix a number $m\in \mc{M}_{D,N}$, we denote $\Q_{D,N,m,\beta}^0:=\Q_{D,N,m}^0\cap \Q_{D,N,\beta}^0$.
Let $t(N,D,m)$ be the number of $\beta\mod 2N$ such that $\Q_{D,N,m,\beta}^0$ is nonempty, that is, $\beta^2 \equiv D\mod 4N$ with $\gcd(N,b,ac)=m$.

\begin{lemma}[{\cite[Proposition p.505--507]{GKZ}}] \label{L:Nbeta} Suppose that $\Q_{D,N,m,\beta}^0$ is non-empty. Fix a decomposition $m=m_1m_2$ with $m_1,m_2>0$ and $\gcd(m_1,m_2)=1$.
We define $$\Q_{D,N,m_1,m_2,\beta}^0:=\{[aN,b,c]\in \Q_{D,N,m,\beta}^0 \mid \gcd(N,b,a)=m_1, \gcd(N,b,c)=m_2 \}.$$
The map $\Q_{D,N,m_1,m_2,\beta}^0 / \Gamma_0(N)$ to $\Q_{D}^0 / \Gamma_0$ induced by 
	$[aN,b,c] \mapsto [aN_1,b,cN_2].$
is a bijection. Here, $N_1N_2$ is any decomposition of $N$ into coprime positive factors satisfying $\gcd(m_1,N_2)=\gcd(m_2,N_1)=1$.
Moreover, there is a decomposition 
	$$\Q_{D,N}^0/\Gamma_0(N) = \bigcup_{\substack{\beta \mod 2N\\ \beta^2 \equiv D \mod 4N}} \Q_{D,N,\beta}^0 /\Gamma_0(N).$$	
\end{lemma}
\noindent The decomposition certainly restricts to $\Q_{D,N,m}^0/\Gamma_0(N)$.
In particular, we have that $|\Q_{D,N,m}^0| = 2^{\sigma_0^+(m)} h(D)$.

Let $j_N^+$ be a Hauptmodul for $\Gamma_0^+(N)$. We define the level-$N$ classical class polynomial
$$H_{D,N,m}(X) = \prod_{Q\in \Q_{D,N,m}^0/\Gamma_0^+(N)} \left(X-j_N^+(\tau_Q)\right).$$
We conjecture the following analogue of Theorem \ref{T:classpoly}.
\begin{conjecture} \label{conj:HDN} The polynomial $H_{D,N,m}(X)$ belongs to $\mb{Z}[X]$ and is irreducible.
\end{conjecture}

\noindent The conjecture would imply that the singular moduli $j_N^+(\tau_Q) \in \mb{Q}$ if and only if $j_N^+(\tau_Q) \in \mb{Z}$
if and only if $Q$ is the unique orbit in $\Q_{D,N,m}^0/\Gamma_0^+(N)$.
In particular, if $j_N^+(\tau_Q) \in \mb{Q}$, then by Lemma \ref{L:Nbeta} we have that
\begin{equation} \label{eq:sm} t(N,D,m) 2^{\sigma_0^+(m)} h(D) \leq 2^{\sigma_0^+(N)}. \end{equation}

\begin{remark} \label{r:smalllist}
In all our cases (see Table \ref{tab:forms}), the number $\sigma_0^+(N)$ of prime divisors of $N$ is no greater than $2$.
So by Lemma \ref{L:Nbeta}, to search for the discriminant $D$ such that $\Q_{D,N}^0$ has only one $\Gamma_0^+(N)$-orbit, it suffices to look at those $D$ with $Cl(D)\cong \mb{Z}_2^g$ ($g=0,1,2$). 
This already reduces the search to a (reasonably small) finite list.
Moreover, if $N$ is prime, $g$ has to be $0$ or $1$.
Due to the presence of $t(N,D,m) 2^{\sigma_0^+(m)}$, the condition is actually more restrictive.
In view of \cite[Proposition 7.1]{Sc}, this implies that singular $K3$ surfaces arising from special fibres of these $25$ families cannot cover all weight-$3$ newforms with rational coefficients. The geometric realization of all such newforms was accomplished in \cite{ES}.
\end{remark}

\begin{example} \label{ex:QDN} Let $N=14$. If $D=-84$, then there are 4 classes in $Cl(D)$ represented by elements 
	$$Q_1=[70,14,1],Q_2=[42,42,11],Q_3=[154,42,3],Q_4=[14,14,5].$$  %
	Since $W_2(Q_2)=W_7(Q_3)=W_{14}(Q_4)=Q_1$ in $\Q_{D,N}^0/\Gamma_0(N)$,
	they reduce to a single element in $\Q_{D,N}^0/\Gamma_0^+(N)$.  
	
	If $D=-20$, then $\Q_{D,N}^0 =  \Q_{D,N,6}^0 \cup \Q_{D,N,-6}^0$, which has representative 
	$$Q_1=[14,6,1],Q_2=[42,-22,3], \text{ and } Q_3=[42,22,3], Q_4=[14,-6,1].$$
	We have that $W_2(Q_2) = W_7(Q_3)= W_{14}(Q_4)=Q_1$ in $\Q_{D,N}^0/\Gamma_0(N)$.
	So $\Q_{D,N}^0/\Gamma_0^+(N)$ has only one element.
\end{example}

\section{Relating to $L$-functions} \label{S:Lfun}
\subsection{Computing $d^2H_d(\tau)$}
Let $Q=[a,b,c]\in \Q_D$.
Throughout we assume that $d \mid a$ and $D=b^2-4ac<0$.
Let $\tau=\tau_Q=\frac{-b+\sqrt{D}}{2a}$. 
Recall the series \eqref{eq:Hd} for $H_d(\tau)$.
We have that
\begin{align*} 
d^2 H_d(\tau) 
&= \sum_{m,n}'  \frac{4d^2(dm\frac{b}{2a}+n)^2}{(n^2+\frac{b}{a}dmn+\frac{c}{a}d^2m^2)^3}  - \sum_{m,n}' \frac{d^2}{(n^2+\frac{b}{a}dmn+\frac{c}{a}d^2m^2)^2}, \\
&= aD \sum_{m,n}'  \frac{dm^2}{(\frac{a}{d}n^2+bmn+cdm^2)^3} + \sum_{m,n}'  \frac{3a^2}{(\frac{a}{d}n^2+bmn+cdm^2)^2}. 
\end{align*}

We have a simple but important observation, which can be verified by straightforward calculation.
\begin{observation} For $Q=[\frac{a}{d},b,cd]$, $x=\frac{2a}{D}$ is the unique $x$ such that $[0,0,d]+xQ$ is $Q$-spherical.
\end{observation}

\noindent From this observation, we continue our calculation.
\begin{align} \label{eq:1}
d^2 H_d(\tau) &= aD \sum_{m,n}'  \frac{dm^2+x(\frac{a}{d}n^2+bmn+cdm^2)}{(\frac{a}{d}n^2+bmn+cdm^2)^3} + (3a^2 - xaD) \sum_{m,n}'  \frac{1}{(\frac{a}{d}n^2+bmn+cdm^2)^2}, \notag \\
&= a \sum_{m,n}'  \frac{db^2m^2+2a (\frac{a}{d}n^2+bmn-cdm^2)}{(\frac{a}{d}n^2+bmn+cdm^2)^3} + \sum_{m,n}'  \frac{a^2}{(\frac{a}{d}n^2+bmn+cdm^2)^2}. 
\end{align}

To proceed further, we focus on the cases $b=ka$ or $b=kc$ ($k\in\mb{Z}$), that is, the cases covered by Lemma \ref{L:ka} (Remark \ref{r:kc}) and thus Proposition \ref{P:CM3theta}.
For $b=ka$, we have that
\begin{align*}
d^2 H_d(\tau) &= a \sum_{m,n}'  \frac{d(ka)^2m^2+2a (\frac{a}{d}n^2+kamn-cdm^2)}{(\frac{a}{d}n^2+kamn+cdm^2)^3} + \sum_{m,n}'  \frac{a^2}{(\frac{a}{d}n^2+kamn+cdm^2)^2}, \\
&= a \sum_{m,n}'  \frac{ 2a(\frac{a}{d}n^2 - cdm^2)}{(\frac{a}{d}n^2 + kamn + cdm^2)^3} + \sum_{m,n}' \frac{a^2}{(\frac{a}{d}n^2+ kamn + cdm^2)^2}, \hspace{1.9cm} \text{(Lemma \ref{L:ka})} \\
&=a^2 \left( 4 L(\Theta_{Q_{/d}},3) + \zeta_{Q_{/d}}(2) \right).
\end{align*}
Here, for $Q=[a,b,c]$ with $d\mid a$, we denote $Q_{/d}:=[\frac{a}{d},b,cd]$. So we obtain the following.

\begin{proposition} \label{P:dHd} If $\tau=\tau_Q$ for $Q=[a,ka,c]$ and $d\mid a$, then
\begin{equation} \label{eq:ka} \frac{\Img  \tau}{(2\pi)^3} d^2 H_d(\tau) = \frac{a\sqrt{D}}{(2\pi)^3} \Big( 2 L(\Theta_{Q_{/d}},3) + \frac{1}{2}\zeta_{Q_{/d}}(2) \Big).\end{equation}
\end{proposition}

\begin{remark} \label{r:dHd} If $Q=[a,kc,c]$ with $d\mid k$, it is easy to check by Remark \ref{r:kc}.(2) that we have the same equality \eqref{eq:ka}.
\end{remark}

\begin{corollary} \label{C:dHd} Suppose that the form class group $Cl(D)$ is isomorphic to $\mb{Z}_2^g$ where $D=m^2d_K$. 
Suppose that for some $Q=[a,ka,c]$, $\{Q_{/d}\}_{d\in \delta}$ constitutes a set of representatives in $Cl(D)$ for a set $\delta$ of divisors of $a$.
	Then for any character $\phi: Cl(D) \to \{\pm 1\}$, we have that 
	$$\frac{\Img  \tau}{(2\pi)^3} \sum_{d\in \delta}\phi(Q_{/d})  d^2 H_d(\tau) = \frac{a\sqrt{D}}{(2\pi)^3} \Big(2L(f_{\phi,3},3) + \frac{1}{2}L(f_{\phi,1},2)\Big).$$
\end{corollary}
\begin{proof} By Lemma \ref{L:ka} all $Q_{/d}$ satisfy the conditions of Proposition \ref{P:CM3theta}. So $f_{\phi,3}=\sum_{d\in\delta} \phi(Q_{/d})\Theta_{Q_{/d}}$ and $f_{\phi,1}=\sum_{d\in\delta} \phi(Q_{/d})\theta_{Q_{/d}}$.
Then the result follows from Proposition \ref{P:dHd}.
\end{proof}

Another important case is when $a=cde$.
For this case, we shall compute $-d^2 H_d(\tau) + e^2 H_e(\tau)$ rather than a single summand. 
We have from \eqref{eq:1} that
\begin{align*}  -d^2 H_d(\tau) + e^2 H_e(\tau) &= -a \sum_{m,n}'\frac{d^{-1}(dbm+2an)^2}{(\frac{a}{d}n^2+bmn+cdm^2)^3} + a \sum_{m,n}'\frac{e^{-1}(ebm+2an)^2}{(\frac{a}{e}n^2+bmn+cem^2)^3}, \\
&= -a \sum_{m,n}' \frac{(db^2-4a^2/e)m^2-(eb^2-4a^2/d)n^2}{(cen^2+bmn+cdm^2)^3}, \\
&= -aD \sum_{m,n}' \frac{dm^2-en^2}{(cen^2+bmn+cdm^2)^3},\\
&= \frac{aD}{c}  2 L(\Theta_{Q_{/d}},3).
\end{align*}
So we obtain the following.
\begin{proposition} \label{P:deH} If $\tau=\tau_Q$ for $Q=[cde,b,c]$, then
	$$\frac{\Img \tau}{(2\pi)^3} \left(-d^2 H_d(\tau) + e^2 H_e(\tau)\right) = c^{-1} \left(\frac{\sqrt{D}}{2\pi}\right)^3 L(\Theta_{Q_{/d}},3).$$
\end{proposition}

\subsection{Lists} \label{ss:list}
To better state our next main result, we will rescale some special values of the relevant $L$-functions.
Let $g_d$ be a newform of weight $3$ and level $d$. We set
\begin{equation} \label{eq:resc1} \Lt(g_{d},3) := \frac{2\sqrt{d}}{(2\pi)^3}  L(g_d,3) = \frac{\ep}{d}L'(g_d,0), \end{equation}
where the latter equality follows from the functional equation of $L(g_d,s)$ and $\ep\in \{\pm 1\}$ is the {\em sign} of the functional equation.
We also denote
\begin{equation} \label{eq:resc2}  l_{d_1,d_2}: = \frac{\sqrt{-d_1d_2}}{(2\pi)^3}    L(\chi_{d_1},2) L(\chi_{d_2},2). \end{equation}
It is well-known that if $d>0$ and the conductor of $\chi_d$ is $d_0$, then $\frac{L(\chi_d,2)}{\sqrt{d_0} \pi^2}$ is a rational number, which can be easily calculated;
if $d<0$ and $\chi_d$ is primitive, then $L(\chi_d,2) = \frac{4\pi}{(-d)^{3/2}} L'(\chi_d, -1).$
Note that if $d$ is a fundamental discriminant, then $\chi_d$ is primitive.
Then it is easy to see that Theorem \ref{T:mmsv} below is equivalent to Theorem \ref{T:intro2}.

\begin{theorem} \label{T:mmsv} For all the $25$ families except for $V_2$ and $B_1$, and all known rational singular moduli of $c(\tau)$ of discriminant $D$,
	the value of $\Re(\mtd(c(\tau)))$ is equal to 
	$$\alpha \Lt(g_{d},3) + \beta l_{d_1,d_2},$$
for some newform $g_{d}$ of weight $3$ with rational coefficients, some fundamental discriminants $d_1$ and $d_2$ with $d_1d_2=D$, and $\alpha,\beta\in \mb{Q}$.
Moreover, $-D/d$ is a square.
The complete list is given below.
\end{theorem}

\begin{proof} The proofs for all entries in the lists are similar. We start with the Eisenstein-Kronecker series in Theorem \ref{T:mmmLG}, then use Proposition \ref{P:dHd} (Remark \ref{r:dHd}) or Proposition \ref{P:deH} to relate $L$-functions. 
In this way we cover all but $6$ cases ($2$ rational $c(\tau)$ and $4$ quadratic $c(\tau)$ marked by $\circledS$ in the remark column). For the $6$ exceptional cases, almost the same calculation goes through. 
	
As an illustration, we prove a moderately complicated (but typical) case:
\begin{align*} \Re(\mtd(-7,V_{28})) &= 14\Lt(g_{84d},3) + 14l_{12,-7}.
\end{align*}
which is the second row in the table for $V_{28}$.
By the second last row of Table \ref{tab:forms} and Theorem \ref{T:mmmLG}, we need to compute 
$$\frac{\Img \tau}{(2\pi)^3} \left(- 1^2 H_1(\tau) -  2^2 H_2(\tau) + 7^2H_7(\tau) + 14^2 H_{14}(\tau) \right),$$
where $\tau$ is the root of $14x^2+14x+5=0$ in $\mb{H}$. 
We have explained in Example \ref{ex:QDN} why $c(\tau)\in \mb{Z}$, then numerical calculation finds $c(\tau)=-7$.

Since the forms $Q=[14,14,5]$, $Q_{/2}=[7,14,10]$, $Q_{/7}=[2,14,35]$, and $Q_{/14}=[1,14,70]$ constitute a set of representatives in $Cl(-84)$,
by Corollary \ref{C:dHd} the above is equal to
$$\frac{14\sqrt{84}}{(2\pi)^3} \big(2L(f_{\phi,3},3) + \frac{1}{2}L(f_{\phi,1},2)\big).$$
where $\phi$ sends $(Q,Q_{/2},Q_{/7},Q_{/14})$ to $(-1,-1,1,1)$.
We have seen from Example \ref{ex:genusch} that $\phi=\chi_{12,-7}$ so $L(f_{\phi,1},2)=2L(\chi_{12},2)L(\chi_{-7},2)$ by Theorem \ref{T:genusch}.
From Example \ref{ex:g84} we have seen that $f_{\phi,3} = g_{84d}$.
After the rescaling \eqref{eq:resc1} and \eqref{eq:resc2}, we get the desired result.

\end{proof}

\begin{remark} As we shall see below that in most cases we have that $\alpha=\beta$ and $D=d_1d_2=d$.
If this is not the case, then either the quadratic form $Q$ is not primitive or the relevant $g_D$ is not a newform so that a relation in Proposition \ref{P:tonew} is invoked. 	
We will indicate this in the remark columns.
\end{remark}

Before displaying the list, we make a few other comments.
\begin{enumerate}
\item For the families $V_2$ and $B_1$, it seems hard to conjecture a relation at any of the $13$ rational singular moduli. So far the only conjectured relation is $m(1728, V_2) = \frac{1}{2} L'(g_{144a},0)$ (\cite[Conjecture 4.11]{Sa0}). This conjecture remains a big challenge for us.
\item The lists below settle all the conjectures for $V_4,V_6$, and $V_8$ in \cite{Sa0} at rational singular moduli.
Quadratic singular moduli can be treated similarly. As an illustration, we give such results at all {\em regular singular points} of Picard-Fuchs equations. 
We label the regular singular points with a `$*$' in the column $c(\tau)$.
\item The results for $\Re(\mtd(\tau))$ between every pair of double lines are not guaranteed to be equal to $m\left(f_{c(\tau)}\right)$ by Lemma \ref{L:mtd} (cf., Remark \ref{r:regionK}).
\end{enumerate}

\renewcommand{\arraystretch}{1.11} 

\begin{minipage}{15cm}
\paragraph{$$\bf V_4$$} 
$$\begin{array}{|c|c|c|c|} \hline
	\tau & c(\tau) & \Re(\mtd(\tau)) & \text{remark}\\ \hline
		[2,2,19]& -2^{10}3^47^4 & 160\Lt(g_{148a},3) + 160l_{37,-4}  & \\ \hline 
		[2,2,13]& -2^{14}3^45  & 160\Lt(g_{100a},3) + 160l_{5,-20} &\\ \hline 
		[2,2,7] & -2^{10}3^4  & 160\Lt(g_{52a},3)+160l_{13,-4} & \text{\cite[Conj]{Sa0}}\\ \hline 
		[2,2,5] & -12288  & 160\Lt(g_{36a},3)+160l_{12,-3} & \text{\cite[Conj]{Sa0}}\\ \hline 
		[1,1,2]	&  -3969 & 400\Lt(g_{7},3)+ 80 l_{4,-7}  &  \text{\cite[Conj]{Sa0}}, \ref{P:tonew} \\ \hline
		[2,2,3] & -1024  &  160\Lt(g_{20a},3)+160 l_{5,-4} & \text{\cite[Conj]{Sa0}}\\ \hline
		[1,1,1]	&  -144   &  160\Lt(g_{12},3)+80 l_{4,-3} & \text{\cite[Conj]{Sa0}}\\	\hline \hline
		[2,2,1]	&  0   &  0 &\\	\hline
		[2,1,1] & 81   & 280\Lt(g_{7},3) &  \text{\cite[Conj]{Sa0}}\\ \hline \hline
		[2,0,1] & 256*  & 320\Lt(g_8,3) &  \text{\cite[Conj]{Sa0}}\\ \hline 
		[1,0,1]&  648 & 160\Lt(g_{16},3)+80l_{4,-4} &  \text{\cite[Conj]{Sa0}}\\ \hline
		[2,0,3]& 2304  & 160\Lt(g_{24a},3)+160l_{8,-3} & \text{\cite{Sa0}} \\ \hline 
		[2,0,5]& 20736  & 160\Lt(g_{40a},3)+160l_{5,-8} & \text{\cite{Sa0}}\\ \hline 
		[2,0,9]& 2^87^4  & \frac{1600}{3}\Lt(g_{8},3)+160l_{24,-3} & \text{\cite{Sa0}, \ref{P:tonew}}\\ \hline 
		[2,0,11]& 2^83^411^2  & 160\Lt(g_{88a},3)+160l_{8,-11} & \\ \hline 
		[2,0,29]& 2^83^811^4  & 160\Lt(g_{232a},3)+160l_{29,-8} &\\ \hline
\end{array}$$
\end{minipage}

\vspace{1cm}
\begin{minipage}{15cm}
\paragraph{$$\bf V_6$$}
$$\begin{array}{|c|c|c|c|} \hline
\tau & c(\tau) & \Re(\mtd(\tau)) & \text{remark}\\ \hline
[3,3,23]& -2^63^35^6   & 90\Lt(g_{267a},3) +90l_{89,-3} &\\ \hline 
[3,3,13]& -2^63^67 &  90\Lt(g_{147a},3) +90l_{21,-7} & \\ \hline 
[3,3,11]& -2^{12}3^3  & 90\Lt(g_{123a},3) +90l_{41,-3} &\\ \hline
[3,3,7] & -8640  & 90\Lt(g_{75a},3) +90l_{5,-15} &\\ \hline 
[3,3,5] & -1728  & 90\Lt(g_{51a},3) +90l_{17,-3} &\\ \hline 
[1,1,1]	&  -192 & 90\Lt(g_{27},3) +60l_{9,-3} &\\ \hline
[3,3,2] & -27  & 90\Lt(g_{15b},3) +90l_{5,-3} &\\ \hline \hline 
[3,3,1]	&  0 & 0 &\\	\hline
[3,2,1]& 8  & 120\Lt(g_{8},3) & \text{\cite[Conj]{Sa0}}\\ \hline 
[3,1,1]&  64 & 165\Lt(g_{11},3) &\\ \hline \hline
[3,0,1]& 108*  & 180\Lt(g_{12},3) & \text{\cite[Conj]{Sa0}} \\ \hline 
[3,0,2]& 216  &  90\Lt(g_{24b},3) +90l_{8,-3} & \text{\cite{Sa0}} \\ \hline 
[3,0,4]& 1458  & \frac{405}{2}\Lt(g_{12},3) +180l_{12,-4} & \text{\cite{Sa0}, \ref{P:tonew}} \\ \hline
[3,0,5]& 3375  & 180\Lt(g_{15b},3) + 117l_{20,-3} & \text{\cite[Conj]{Sa0}, \ref{P:tonew}}\\ \hline
\end{array}$$
\end{minipage}

\vspace{1cm}
\begin{minipage}{15cm}
\paragraph{$$\bf V_8$$}
$$\begin{array}{|c|c|c|c|} \hline
\tau & c(\tau) & \Re(\mtd(\tau)) & \text{remark}\\ \hline
[4,4,5] & -512   &  64\Lt(g_{64},3) + 64l_{8,-8}& \text{\cite[Conj]{Sa0}}\\ \hline 
[4,4,3] & -64  & 64\Lt(g_{32},3) + 64l_{8,-4} & \text{\cite[Conj]{Sa0}}\\ \hline
[2,2,1]	&  -8 &  64\Lt(g_{16},3) + 32l_{4,-4} &\text{\cite[Conj]{Sa0}}\\	\hline 
[4,3,1] & 1  &  56 \Lt(g_{7},3) &\text{\cite[Conj]{Sa0}}\\ \hline 
[4,2,1]&  16 &  96 \Lt(g_{12},3) &\text{\cite[Conj]{Sa0}}\\ \hline 
[4,0,1]& 64*  & 128 \Lt(g_{16},3) & \text{\cite{Sa0}}\\ \hline 
[4,0,3]& 256  & 64 \Lt(g_{48},3) + 64l_{12,-4} & \text{\cite{Sa0}}\\ \hline
[4,0,7]& 4096  & 64 \Lt(g_{112},3) + 64l_{28,-4} & \text{\cite[Conj]{Sa0}}\\ \hline
\end{array}$$
\end{minipage}

\vspace{1cm}
\begin{minipage}{15cm}
\paragraph{$$\bf V_{10}$$}
$$\begin{array}{|c|c|c|c|} \hline
\tau & c(\tau) & \Re(\mtd(\tau)) & \text{remark}\\ \hline
[5,5,13] & -15228  &  50\Lt(g_{235a},3)+ 50 l_{5,-47} &\\ \hline 
[5,5,7] & -828   &  50\Lt(g_{115a},3) + 50 l_{5,-23} &\\ \hline 
[5,5,3] & -28  & 50\Lt(g_{35b},3) + 50 l_{5,-7}&\\ \hline
[5,5,2]	&  -3 &  50\Lt(g_{15a},3) + 50l_{5,-3}&\\	\hline \hline
[10,10,3]& 22-10\sqrt{5} * & 100(\Lt(g_{20a},3)-\Lt(g_{20b},3)) & \\ \hline 
[5,4,1]&  0 &  0 &\\ \hline
[5,3,1]& 4  & 55\Lt(g_{11},3)&\\ \hline
[5,2,1]& 18  & 80\Lt(g_{16},3)&\\ \hline
[5,1,1]& 36  & 95\Lt(g_{19},3)&\\ \hline \hline
[5,0,1]& 22+10\sqrt{5} * & 50(\Lt(g_{20a},3)+\Lt(g_{20b},3)) &\\ \hline 
[5,0,2]& 72  &  50\Lt(g_{40b},3) + 50 l_{5,-8} &\\ \hline 
[5,0,3]& 147  &  125\Lt(g_{15a},3) + 65 l_{20,-3} & \ref{P:tonew} \\ \hline 
\end{array}$$
\end{minipage}

\vspace{1cm}
\begin{minipage}{15cm}
\paragraph{$$\bf V_{12}$$}  
$$\begin{array}{|c|c|c|c|} \hline
\tau & c(\tau) & \Re(\mtd(\tau)) & \text{remark}\\ \hline
[3,3,1] &  -1 & 36\Lt(g_{12},3) + 24l_{4,-3} &\\	\hline  \hline
[30,24,5] &  17-12\sqrt{2}* & 48 \Lt(g_{24a},3)- 36 \Lt(g_{24b},3)& \\	\hline
[6,4,1] &  1 & 32\Lt(g_{8},3) &\\	\hline \hline
[6,0,1] &  17+12\sqrt{2}* & 48 \Lt(g_{24a},3)+ 36 \Lt(g_{24b},3) &\\	\hline 
\end{array}$$
\end{minipage}

\vspace{1cm}
\begin{minipage}{15cm}
\paragraph{$$\bf V_{14}$$}
$$\begin{array}{|c|c|c|c|} \hline
\tau & c(\tau) & \Re(\mtd(\tau)) & \text{remark}\\ \hline
[7,7,17] & -10648  &  35 \Lt(g_{427a},3)+ 35 l_{61,-7}&\\ \hline 
[7,7,5] & -64   &  35\Lt(g_{91b},3) + 35 l_{13,-7} &\\ \hline 
[7,7,3] & -8  & 35\Lt(g_{35a},3) + 35 l_{5,-7} &\\ \hline \hline 
[7,7,2]	&  -1* & 70\Lt(g_{7},3) &\\	\hline 
[7,5,1]&  0 &  0 &\\ \hline
[7,4,1]& 2  & 30\Lt(g_{12},3) &\\ \hline
[7,3,1]& 8  & \frac{95}{2}\Lt(g_{19},3)&\\ \hline
[7,1,1]& 24  & \frac{135}{2}\Lt(g_{27},3)&\\ \hline \hline
[7,0,1]& 27*  & 210 \Lt(g_{7},3) & \ref{P:tonew} \\ \hline 
[7,0,4]& 125  & \frac{455}{2}\Lt(g_{7},3) +35 l_{28,-4} & \ref{P:tonew}\\ \hline 
\end{array}$$
\end{minipage}

\vspace{1cm}
\begin{minipage}{15cm}
\paragraph{$$\bf V_{16}$$}
$$\begin{array}{|c|c|c|c|} \hline
\tau & c(\tau) & \Re(\mtd(\tau)) & \text{remark}\\ \hline
[8,8,3]	& -4  &  80\Lt(g_{8},3) + 32l_{8,-4} & \ref{P:tonew}\\	\hline \hline
[24,16,3]	&  12-8\sqrt{2}* & 32\Lt(g_{32},3) - 80\Lt(g_{8},3) & \ref{P:tonew}\\	\hline 
[8,5,1]	& 1  & 21\Lt(g_7,3) &\\	\hline 
[8,4,1]	& 4  & 32\Lt(g_{16},3) &\\	\hline 
[8,2,1]	& 16  &  161\Lt(g_7,3) & \ref{P:tonew}\\	\hline  \hline
[8,0,1]	&  12+8\sqrt{2}* & 32\Lt(g_{32},3) + 80\Lt(g_{8},3) & \ref{P:tonew}\\	\hline 
\end{array}$$
\end{minipage}

\vspace{1cm}
\begin{minipage}{15cm}
\paragraph{$$\bf V_{18}$$} 
$$\begin{array}{|c|c|c|c|} \hline
\tau & c(\tau) & \Re(\mtd(\tau)) & \text{remark}\\ \hline
[9,9,5]	& -27  &  27\Lt(g_{99},3) + 27 l_{33,-3} &\\ \hline
[3,3,1]	& -3  & 27\Lt(g_{27},3) + 18 l_{9,-3} &\\	\hline \hline
[18,18,5] &  9-6\sqrt{3}* & 54 (L(g_{36a},3)-L(g_{36b},3)) &\\	\hline 
[9,8,2]	&  -1 & 48 \Lt(g_8,3) &\\	\hline 
[9,5,1]	&  1 &  \frac{33}{2} \Lt(g_{11},3) &\\	\hline 
[9,3,1]	&  9 & \frac{81}{2} \Lt(g_{27},3) &\\	\hline \hline
[9,0,1]	&  9+6\sqrt{3}*  &  27 (L(g_{36a},3)+L(g_{36b},3)) &\\	\hline 
[9,0,2]	&  27 & 27\Lt(g_{72},3) + 27 l_{24,-3} &\\	\hline 
\end{array}$$
\end{minipage}

\vspace{1cm}
\begin{minipage}{15cm}
\paragraph{$$\bf V_{22}$$}
$$\begin{array}{|c|c|c|c|} \hline
\tau & c(\tau) & \Re(\mtd(\tau)) & \text{remark}\\ \hline
[11,11,7]	& -44  &  22\Lt(g_{187b},3) + 22 l_{17,-11} &\\	\hline
[11,11,5]	& -12  &  \frac{242}{3}\Lt(g_{11},3) + 22 l_{33,-3} & \ref{P:tonew}\\	\hline \hline
[11,11,3]	& 0  &  44\Lt(g_{11},3) &\\	\hline 
[11,9,2]	& 1  & 28\Lt(g_7,3) & \circledS\\	\hline 
[11,6,1]	& 2  & 8\Lt(g_{8},3) &\\	\hline 
[11,5,1]	& 4  & 19\Lt(g_{19},3) &\\	\hline 
[11,4,1]	& 7  & 84\Lt(g_{7},3) & \ref{P:tonew}\\	\hline 
[11,1,1]	& 16  & 43\Lt(g_{43},3) & \\	\hline 
[33,22,4] & \br{\alpha}_2 * &  \frac{22}{3}(-\frac{45}{2}\Lt(g_{11},3)+(5-\sqrt{33})\Lt(h_{44a},3)+(5+\sqrt{33})\Lt(h_{44a}',3)) & \ref{P:tonew},\circledS\\	\hline  
[44,22,3] & \alpha_2 * & 88(\Lt(h_{44a},3)+\Lt(h_{44a}',3)) \footnote{where $\alpha_1, \alpha_2, \br{\alpha}_2$ are the roots of $x^3 - 20x^2 + 56x - 44$.
	The coefficient fields of $h_{44a}$ and its Galois conjugate $h_{44a}'$ are $\mb{Q}(\sqrt{33})$, and their full LMFDB labels are $44.3.d.a$.} & \circledS\\	\hline \hline 
[11,0,1]	& \alpha_1*  &  88\Lt(g_{11},3) & \ref{P:tonew}\\	\hline 
[11,0,2]	& 22  &  22\Lt(g_{88b},3) + 22l_{8,-11} &\\	\hline 
\end{array}$$
\end{minipage}

\vspace{1cm}
\begin{minipage}{15cm}
\paragraph{$$\bf V_{12a}$$}  
$$\begin{array}{|c|c|c|c|} \hline
\tau & c(\tau) & \Re(\mtd(\tau)) & \text{remark}\\ \hline
[6,6,31]& -1123600  & 36 \Lt(g_{708c},3) +36l_{177,-4}&\\	\hline 
[6,6,17]& -24304  &  36 \Lt(g_{372c},3) +36l_{12,-31}&\\ \hline 
[6,6,11]& -2704  &  36 \Lt(g_{228c},3) +36l_{57,-4} &\\	\hline 
[6,6,7]	& -400  &  36 \Lt(g_{132c},3) +36l_{33,-4} &\\	\hline 
[6,6,5]	& -112  & 36 \Lt(g_{84c},3) +36l_{12,-7} &\\	\hline 
[3,3,2]	& -49  &  90 \Lt(g_{15b},3) + 54l_{20,-3} & \ref{P:tonew}\\	\hline 
[2,2,1]	& -16  & 36 \Lt(g_{36b},3) + 24l_{9,-4} &\\	\hline 
[3,3,1]	& -4*  & 72\Lt(g_{12},3) &\\	\hline  
[6,4,1]	& 0  &  0 &\\	\hline 
[6,3,1]	& 5  & 45 \Lt(g_{15b},3) &\\	\hline 
[6,2,1]	& 16  & 60 \Lt(g_{20b},3) &\\	\hline 
[6,0,1]	& 32*  & 72 \Lt(g_{24b},3) & \\	\hline 
[3,0,1]	& 50  &  117\Lt(g_{12},3) + 36l_{12,-4} &\ref{P:tonew}  \\	\hline 
[2,0,1]	& 96  &  36\Lt(g_{72},3) + 24l_{9,-8} & \\	\hline 
[6,0,5]	& 320  & 36 \Lt(g_{120c},3) + 36l_{8,-15} &\\	\hline 
[6,0,7]	& 896  & 36 \Lt(g_{168c},3) + 36l_{24,-7} &\\	\hline 
[6,0,13]& 10400  &  36 \Lt(g_{312c},3) + 36l_{8,-39} &\\	\hline 
[6,0,17]& 39200  &  36 \Lt(g_{408c},3) + 36l_{17,-24} &\\	\hline 
\end{array}$$
\end{minipage}

\begin{minipage}{15cm}
\paragraph{$$\bf V_{12b}$$} 
$$\begin{array}{|c|c|c|c|} \hline
\tau & c(\tau) & \Re(\mtd(\tau)) & \text{remark}\\ \hline
[6,6,31]&  -1123596 &  48 L(g_{708b},3) + 48l_{12,-59} &\\ \hline
[6,6,17]&  -24300 &  48 L(g_{372b},3) + 48l_{124,-3} &\\	\hline 
[6,6,11]&  -2700 & 48 L(g_{228b},3) + 48l_{76,-3} &\\	\hline 
[6,6,7]	& -396  & 48 L(g_{132b},3) + 48l_{12,-11} &\\	\hline
[6,6,5]	& -108  &  48 L(g_{84b},3) + 48l_{28,-3}&\\	\hline 
[3,3,2]	& -45  &  96 L(g_{15a},3) + 24l_{4,-15} &\ref{P:tonew} \\	\hline 
[2,2,1]	& -12  & 48 L(g_{36},3) + 48l_{12,-3} &\\	\hline 
[3,3,1]	& 0  & 96l_{4,-3} &\\	\hline  
[6,4,1]	&  4* & 64 \Lt(g_{8},3) &\\	\hline 
[6,3,1]	&  9 & 60 \Lt(g_{15a},3) &\\	\hline 
[6,2,1]	&  20 &  80 \Lt(g_{20a},3)&\\ \hline 
[6,0,1]	&  36* &  96 \Lt(g_{24a},3) &\\	\hline  
[3,0,1]	&  54 &  48\Lt(g_{48a},3) + 33 l_{16,-3}&\\	\hline 
[2,0,1]	&  100 & 192\Lt(g_{8},3) + 48 l_{24,-3} &\ref{P:tonew}\\	\hline 
[6,0,5]	&  324 & 48\Lt(g_{120b},3) + 48 l_{40,-3} &\\	\hline 
[6,0,7]	&  900 & 48\Lt(g_{168b},3) + 48 l_{21,-8} &\\	\hline 
[6,0,13]&  10404 & 48\Lt(g_{312b},3) + 48 l_{13,-24} &\\	\hline 
[6,0,17]&  39204 & 48\Lt(g_{408b},3) + 48 l_{136,-3} &\\	\hline 
\end{array}$$
\end{minipage}

\vspace{1cm}
\begin{minipage}{15cm}
\paragraph{$$\bf V_{20}$$}  
$$\begin{array}{|c|c|c|c|} \hline
\tau & c(\tau) & \Re(\mtd(\tau)) & \text{remark}\\ \hline
[10,10,11] & -324  & 20\Lt(g_{340c},3) + 20l_{17,-20} &\\	\hline 
[10,10,7] & -64  & \frac{640}{9}\Lt(g_{20b},3) + 20l_{12,-15} & \ref{P:tonew}\\	\hline  
[2,2,1]	& -20  & 20\Lt(g_{100b},3)+16l_{25,-4}  &\\	\hline 
[5,5,2]	& -9  &  60\Lt(g_{15a},3)+30l_{20,-3}	&\ref{P:tonew}\\\hline 
[10,10,3] & -4*  & 40\Lt(g_{20b},3) &\\	\hline 
[5,4,1]	& -2  &  32\Lt(g_{16},3) &\\	\hline 
[10,6,1] & 0  & 0 &\\ \hline 
[10,5,1] & 1  & 15\Lt(g_{15a},3) &\\	\hline 
[10,4,1] & 4  & 24\Lt(g_{24b},3) &\\	\hline 
[10,2,1] &  12 & 36\Lt(g_{36b},3) &\\  \hline 
[10,0,1] & 16*  & 40\Lt(g_{40b},3) &\\	\hline 
[10,0,3] & 36  &  20\Lt(g_{120d},3) + 20l_{8,-15} &\\	\hline 
[10,0,7] & 196  & 20\Lt(g_{280c},3) + 20l_{40,-7} &\\ \hline 
[10,0,13] & 1296  & 20\Lt(g_{520c},3) + 20l_{65,-8} &\\	\hline 
[10,0,19] & 5776  & 20\Lt(g_{760c},3) + 20l_{8,-95} &\\	\hline 
\end{array}$$
\end{minipage}

\vspace{1cm}
\begin{minipage}{15cm}
\paragraph{$$\bf V_{24}$$} 
$$\begin{array}{|c|c|c|c|} \hline
\tau & c(\tau) & \Re(\mtd(\tau)) & \text{remark}\\ \hline
[12,12,7]	& -32  &  96\Lt(g_{12},3) + 24L_{48,-4} &\ref{P:tonew}\\	\hline 
[12,12,5]	& -8  & 36\Lt(g_{24b},3) + 24L_{8,-12} &\ref{P:tonew}\\	\hline 
[3,3,1]	& -2  & 42\Lt(g_{12},3) + 24L_{12,-4} &\ref{P:tonew}\\	\hline 
[12,9,2]	& 1  & 150l_{5,-3} & \text{\cite[$Q_{-3}$]{Be2}}, \circledS \\	\hline 
[12,6,1]	& 4*  &  24\Lt(g_{12},3) & \text{\cite[$Q_0$]{Be1}}\\	\hline 
[12,4,1]	& 8  &   80\Lt(g_{8},3) & \ref{P:tonew}\\	\hline  
[12,0,1]	& 16*  &  96\Lt(g_{12},3)  & \text{\cite[$Q_{12}$]{Be1}},\ref{P:tonew}\\	\hline 
[12,0,5]	& 64  &  90\Lt(g_{15b},3) + 24L_{20,-12} & \ref{P:tonew} \\	\hline 
\end{array}$$
\end{minipage}

\vspace{1cm}
\begin{minipage}{15cm}
\paragraph{$$\bf V_{28}$$} 
$$\begin{array}{|c|c|c|c|} \hline
\tau & c(\tau) & \Re(\mtd(\tau)) & \text{remark}\\ \hline
[14,14,13] & -175 & 14\Lt(g_{532c},3) + 14l_{76,-7} &\\	\hline
[14,14,5] & -7 & 14\Lt(g_{84d},3) + 14l_{12,-7} &\\	\hline \hline
[7,7,2] & -4*  & 98\Lt(g_{7},3) &\ref{P:tonew} \\	\hline 
[14,12,3] & -3  & 24\Lt(g_{24b},3) &\\	\hline 
[7,5,1]	& -1 & 12\Lt(g_{12},3) &\\	\hline 
[42,28,5] & 5-4\sqrt{2}* & 14\sqrt{2}(\Lt(h_{56c},3)-\Lt(h_{56c}',3)) & \circledS\\	\hline 
[14,7,1] & 0  & 7\Lt(g_{7},3) &\\	\hline 
[14,6,1] & 1  &  10\Lt(g_{20b},3) &\\ \hline 
[14,4,1] & 5  &  20\Lt(g_{40b},3) &\\	\hline 
[14,2,1] & 9  & 26\Lt(g_{52b},3) &\\	\hline \hline 
[14,0,1] & 5+4\sqrt{2}*  & 14(\Lt(h_{56c},3)+\Lt(h_{56c}',3)) 
\footnote{The coefficient fields of $h_{56c}$ and its Galois conjugate $h_{56c}'$ are $\mb{Q}(\sqrt{2})$, and their full LMFDB labels are $56.3.h.c$.}&\\	\hline 
[7,0,1]	& 14  & 133\Lt(g_{7},3) + 14l_{28,-4} &\ref{P:tonew}\\	\hline 
[14,0,3] & 21  & 14\Lt(g_{168d},3) + 14l_{24,-7} &\\	\hline 
[14,0,5] & 45  & 14\Lt(g_{280d},3) + 14l_{40,-7} &\\	\hline  
\end{array}$$
\end{minipage}


\vspace{1cm}
\begin{minipage}{15cm}
\paragraph{$$\bf V_{30}$$} 
$$\begin{array}{|c|c|c|c|} \hline
\tau & c(\tau) & \Re(\mtd(\tau)) & \text{remark}\\ \hline
[15,15,17]	& -363  &  15\Lt(g_{795c},3) + 15l_{265,-3} &\\	\hline 
[15,15,13]	& -135  &  15\Lt(g_{555c},3) + 15l_{37,-15}&\\	\hline 
[15,15,11]	& -75  &  15\Lt(g_{435c},3) + 15l_{145,-3} &\\	\hline 
[15,15,7]	& -15  &  15\Lt(g_{195c},3) + 15l_{13,-15} &\\	\hline 
[3,3,1]	   & -3  &  15\Lt(g_{75b},3) + 12l_{25,-3} &\\	\hline \hline 
[15,15,4] & 0 & 30\Lt(g_{15a},3)&\\	\hline 
[15,7,1]	& 1  &  -\frac{11}{2}\Lt(g_{11},3) &\\	\hline
[15,20,7]	& 1-2i*  & 25\big(\Lt(g_{20a},3) + \frac{1}{2}\Lt(g_{20b},3)\big) & \circledS\\	\hline 
[15,10,2]   & 1+2i*  &  25\big(\Lt(g_{20a},3) + \frac{1}{2}\Lt(g_{20b},3)\big) &\\	\hline 
[15,6,1]	& 3  &  12\Lt(g_{24a},3)&\\	\hline 
[15,5,1]	& 5  &  \frac{35}{2}\Lt(g_{35b},3) &\\	\hline  
[15,3,1]	& 9  &  \frac{51}{2}\Lt(g_{51b},3) &\\	\hline \hline 
[15,0,1]	& 12*  &  75\Lt(g_{15a},3) &\ref{P:tonew}\\	\hline 
[15,0,2]	& 15  &  15\Lt(g_{120d},3) + 15l_{40,-3}&\\	\hline 
[15,0,4]	& 30  &  \frac{645}{8}\Lt(g_{15a},3) + 15l_{12,-20} & \ref{P:tonew} \\	\hline  
\end{array}$$
\end{minipage}
\vspace{1cm}

We are in a good position to discuss the modularity problem for $K3$ surfaces.
For an algebraic $K3$ surface $Y$ over $\mb{Q}$, we denote by $Y_p$ the reduction of $Y$ modulo $p$.
The zeta function of $Y_p$ at ``good" primes is of the form
$$Z_{Y_p}(t) = \frac{1}{(1-T)(1-p^2T)P_p(T)},$$
where $P_p(T)$ is a polynomial of degree $22$ and $P_p(0)=1$.
Furthermore, $P_p(T)$ factors as $P_p(T) = Q_p(T)R_p(T)$,
where $Q_p(T)$ and $R_p(T)$ come from the transcendental and algebraic cycles respectively.
Let $\op{T}(Y) = \op{Pic}(Y)^\perp \subset H^2(Y,\mb{Z})$ be the {\em transcendental lattice} of $Y$.
We define $$L(\op{T}(Y),s):= (*) \prod_{p \text{ good}} \frac{1}{Q_p(p^{-s})},$$
($*$) is the product of the Euler factors corresponding to the primes of bad reduction.

If $Y$ is singular, then $\T(Y)$ is a 2-dimensional lattice, which can be expressed through its intersection form. We define the {\em discriminant} of $Y$ as the discriminant of the intersection form.

\begin{theorem}[Livn\'{e}\cite{Li}] Let $Y$ be a singular $K3$ surface defined over $\mb{Q}$ with discriminant $d$. Then there exists a newform $g$ of weight $3$ with CM by
	$\mb{Q}(\sqrt{d})$ such that $L(\op{T}(Y),s)=L(g,s)$.
\end{theorem}

\begin{conjecture} Each above $K3$ hypersurface (after desingularization if necessary) is modular with the weight-3 newform given in the column $\Re(\wtd{m}(\tau))$.
\end{conjecture}

\begin{remark} It is not very difficult to verify one by one as did in \cite{AOP}. But we do not know a conceptional uniform proof.
	We also suspect that some of them, especially those in $V_{12a},V_{12b},V_{20},V_{24},V_{28},V_{30}$ (and $B_{6a}, B_{6b}$), are not singular (i.e., have Picard rank less than $20$).
	In those cases, $Q_p(T)$ would have some trivial factors.
\end{remark}

\section{Exotic Relations} \label{S:rel}
If one family of hypersurface $\tilde{f}_c$ is obtained from another family $f_c$ by pulling back some degree $d$ cover of the base space,
then we have the following relation 
		$$ m(\tilde{f}_c) = \frac{1}{d} m(f_{c^d}).$$
We call such relations {\em trivial}. Other relations among Mahler measures of families are called {\em exotic}.
As a convention in this section, we shall write $m_i(c)$ for $m(f_c)$ where $f$ is the Laurent polynomial labelled by $V_i$ (or $B_i$ if $i=6a,6b$).
However, for the function $\mtd$ we will switch to the modular parameter $\tau$, namely, $\mtd_i(\tau) = \mtd_i(c_i(\tau))$.

M. Rogers proved in \cite[Theorem 2.5]{Ro} the following exotic relations between the Mahler measures of two families:
\begin{align}
m_{6b}\left(3(z+z^{-1})\right) &= \frac{1}{20} m_4\bigg(\frac{9(z^2+3)^4}{z^6}\bigg) + \frac{3}{20} m_4\bigg(\frac{9(3+z^{-2})^4}{z^{-6}}\bigg), \label{eq:relP}\\
m_{24}(z) &= \frac{8}{15} m_6\bigg(\frac{(z-4)^3}{z}\bigg) - \frac{1}{15} m_6\bigg(\frac{(16-z)^3}{z^2}\bigg). \label{eq:relQ}
\end{align}
We shall see that such relations arise naturally from the modular relations among the modular functions.
As the first step, we observe that the following relations are the direct consequence of Theorem \ref{T:mmm} and the column $e(\tau)$ of Table \ref{tab:forms}.
As a trivial remark, we mention that the relations hold unconditionally if we replace $\mtd$ by the right side of the equation \eqref{eq:local}.
\begin{corollary} \label{C:relation} We have the following relations for $\tau$ lying in the fundamental domains containing $i\infty$ of the modular functions $c(\tau)$ for both sides.
\begin{align} 
\mtd_8(\tau) &= \frac{1}{5}\mtd_4(\tau) +\frac{2}{5}\mtd_4(2\tau), \label{eq:q8} \\
\mtd_{12}(\tau) &= \frac{1}{2}\big(\mtd_{12a}(\tau)+\mtd_{12b}(\tau)\big), \label{eq:q12} \\
\mtd_{16}(\tau) &= \frac{1}{20}\mtd_4(\tau)+ \frac{3}{40}\mtd_4(2\tau)+ \frac{1}{5}\mtd_4(4\tau),  \label{eq:q16}\\
\mtd_{18}(\tau) &= \frac{1}{10} \mtd_6(\tau) + \frac{3}{10}\mtd_6(3\tau), \label{eq:q18} \\
\mtd_{12a}(\tau) &= \frac{1}{5}\mtd_6(\tau)+ \frac{2}{5}\mtd_6(2\tau), \label{eq:q12a} \\
\mtd_{12b}(\tau) &= \frac{1}{10}\mtd_4(\tau)+ \frac{3}{10}\mtd_4(3\tau), \label{eq:q12b} \\
\mtd_{20}(\tau) &= \frac{1}{5}\mtd_{10}(\tau)+ \frac{2}{5}\mtd_{10}(2\tau), \label{eq:q20} \\
\mtd_{24}(\tau) & = \frac{1}{15} \mtd_6(\tau)- \frac{1}{15} \mtd_6(2\tau) + \frac{4}{15} \mtd_6(4\tau), \label{eq:q24} \\
\mtd_{28}(\tau) &= \frac{1}{5}\mtd_{14}(\tau) + \frac{2}{5}\mtd_{14}(2\tau), \label{eq:q28}\\
\mtd_{30}(\tau) &= \frac{1}{10}\mtd_{10}(\tau) + \frac{3}{10}\mtd_{10}(3\tau). \label{eq:q30}
\end{align}
\end{corollary}

To convert the above relations to exotic relations among Mahler measures in the parameter $c$, 
we need the relations among the corresponding Hauptmoduln.
Some of the relations in the following lemma may be known in the literature, and most likely appeared as the defining equations of certain modular curves.
We found these relations by numerical experiment with Matlab \cite{Mat}.
To rigorously verify them, we can multiply some appropriate $\eta$-quotient to make both sides holomorphic,
then check the equality using Sturm's theorem.

It is more convenient to switch to a different $\eta$-quotient representation for the following Hauptmoduln:
\begin{align*}
(h+4h^{-1})(h+8h^{-1}),\ h&=\frac{1^24^1}{2^18^2}\ \text{ for } V_{16}, \\
(h+4h^{-1})^2,\ h&=\frac{1^1 3^1}{4^1 12^1}\ \text{ for } V_{24},\\
h^3+8h^{-3}+5,\ h&=\frac{1^1 7^1}{2^1 14^1}\ \text{ for } V_{28}.
\end{align*}

\begin{lemma} \label{L:modeq} We have the following modular equations 
\begin{align}
&z + 16z^{-1} = y &&  y=\frac{2^{24}}{1^{12}4^{12}},\ z=\frac{1^4}{4^4},   \label{eq:p8} \\
&z +9 +27z^{-1} = y && y=\frac{3^{12}}{1^6 9^6},\quad  z= \frac{1^3}{9^3},  \label{eq:p18}\\
&\begin{aligned} \label{eq:p16}
&z =  \frac{y^4}{(y + 4)(y + 8)^2} \\
&z = \frac{y(y + 8)}{y + 4}  
\end{aligned}
&&\begin{aligned}
y&=\frac{2^{4}4^{4}}{1^{4} 8^{4}}, &z&= \frac{1^{24}}{2^{24}},  \\
y&=\frac{2^{4}4^{4}}{1^{4} 8^{4}}, &z&= \frac{2^{12}}{4^{12}}; 
\end{aligned}\\[2ex]
&\begin{aligned}\label{eq:p12a}
&z + 27z^{-1} +54 = (y+16)^3y^{-2}  \\[2ex]
&z + 27z^{-1} +54 = (y+4)^3y^{-1}   
\end{aligned}
&&\begin{aligned}
y&=\frac{1^6 3^6}{2^6 6^6}, &z&=\frac{1^{12}}{3^{12}}, \\
y&=\frac{1^6 3^6}{2^6 6^6}, &z&=\frac{2^{12}}{6^{12}}; 
\end{aligned}\\[2ex]
&\begin{aligned} \label{eq:p12b}
&z + 64z^{-1} = (y^2+27)^2y^{-3} \\[2ex] 
&z + 64z^{-1} = (y^2 + 3)^2y^{-1} 
\end{aligned}
&&\begin{aligned}
y&=\frac{1^2 2^2}{3^2 6^2}, &z&=\frac{1^{12}}{2^{12}}, \\
y&=\frac{1^2 2^2}{3^2 6^2}, &z&=\frac{2^{12}}{4^{12}};
\end{aligned}\\[2ex]
&\begin{aligned} \label{eq:p20}
&z+125z^{-1}+22 = (y+4)(1+8y^{-1})^2 \\[2ex] 
&z+125z^{-1}+22 = (y + 2)^2(1 + 4y^{-1}) 
\end{aligned}
&&\begin{aligned}
y&=\frac{1^4 5^4}{2^4 10^4}, &z&=\frac{1^6}{5^6},  \\
y&=\frac{1^4 5^4}{2^4 10^4}, &z&=\frac{2^6}{10^6}; 
\end{aligned}\\[2ex]
&\begin{aligned} \label{eq:p24}
&z +729z^{-1}+54  = \frac{(y+8)^6}{y^4(y+4)}  \\
&z +27z^{-1}  = \frac{(y + 2)(y - 4)(y + 8)}{y(y + 4)}  \\
&z +729z^{-1}+54  = \frac{(y+2)^6}{y(y+4)}  
\end{aligned}
&&\begin{aligned}
y&=\frac{1^{2}3^{2}}{4^{2} 12^{2}}, &z&= \frac{1^{12}}{3^{12}}, \\
y&=\frac{1^{2}3^{2}}{4^{2} 12^{2}}, &z&= \frac{2^{6}}{6^{6}}, \\
y&=\frac{1^{2}3^{2}}{4^{2} 12^{2}}, &z&= \frac{4^{12}}{12^{12}};
\end{aligned}
\end{align}
\begin{align}
&\begin{aligned} \label{eq:p28}
&z+49z^{-1}+14=(y+4)^4y^{-2}  \\[2ex]
&z+49z^{-1}+14=(y+2)^3y^{-1}  
\end{aligned}
&&\begin{aligned}
y&=\frac{1^3 7^3}{2^3 14^3}, &z&=\frac{1^4}{7^4},  \\
y&=\frac{1^3 7^3}{2^3 14^3}, &z&=\frac{2^4}{14^4};  
\end{aligned}\\[2ex]
&\begin{aligned} \label{eq:p30}
&z+125z^{-1}+22 = (y+9+27y^{-1})^2y^{-1} \\[2ex]
&z+125z^{-1}+22 = (y+3+3y^{-1})^2y 
\end{aligned}
&&\begin{aligned}
y&=\frac{1^25^2}{3^215^2}, &z&=\frac{1^6}{5^6}, \\
y&=\frac{1^25^2}{3^215^2}, &z&=\frac{3^6}{15^6}. 
\end{aligned}
\end{align}

\end{lemma}

Next we translate the condition on $\tau$ in Corollary \ref{C:relation} to the condition on the parameter $c=c(\tau)$.
We need to pick $\tau$ such that $\tau$ lie in the fundamental domain of $c_i(\tau)$ containing $i\infty$ for each term $\mtd_i(c_i(\tau))$ in the relation. 
Note that this is not a problem if each $c_i$ is sufficiently large.
\begin{theorem}	\label{T:rel} For $|y|$ $($or $|z|)$ sufficiently large, we have the following relations on Mahler measures \begin{align}
	m_8\left(z+256z^{-1}+32\right) &= \frac{1}{5}m_4\left(\frac{(z+32)^4}{z^2(z+16)}\right) + \frac{2}{5} m_4\left(\frac{(z+8)^4}{z(z+16)}\right), \label{eq:r8} \\
	m_{12}(y) &= \frac{1}{2} \left(m_{12a}(y+y^{-1}-2) +m_{12b}(y+y^{-1}+2) \right), \label{eq:r12}\\[1.1ex]
    m_{16}(y+32y^{-1}+12) =& \medmath{\frac{1}{20} m_4 \bigg(\frac{(y^2 + 32y + 128)^4}{y^4(y + 4)(y + 8)^2}\bigg) + \frac{3}{40}m_4\left(\frac{(y^2 + 8y + 32)^4}{y^2(y + 4)^2(y + 8)^2}\right) + \frac{1}{5}m_4\left(\frac{(y^2 + 8y + 8)^4}{y(y + 4)^2(y + 8)}\right),} \label{eq:r16}\\[1.1ex]
	m_{18}(z+27z^{-1}+9) &= \frac{1}{10}m_6\left(\frac{(z+9)^6}{z^3(z^2 + 9z + 27)}\right) + \frac{3}{10}m_6\left(\frac{(z+3)^6}{z(z^2 + 9z + 27)}\right), \label{eq:r18} \\
	m_{12a}(y+64y^{-1}+16) &= \frac{1}{5}m_6\left((y+16)^3y^{-2}\right)+ \frac{2}{5}m_6\left((y+4)^3y^{-1}\right), \label{eq:r12a} \\
	m_{12b}(y+81y^{-1}+18) &= \frac{1}{10}m_4\left((y+27)^4y^{-3}\right)+ \frac{3}{10}m_4\left((y + 3)^4y^{-1}\right), \label{eq:r12b} \\
	m_{20}(y+16y^{-1}+8) &= \frac{1}{5}m_{10}\left((y+4)(1+8y^{-1})^2\right)+ \frac{2}{5}m_{10}\left((y + 2)^2(1 + 4y^{-1})\right), \label{eq:r20} \\[1.1ex]
	m_{24}(y+16y^{-1}+8) &= \medmath{\frac{1}{15}m_6\left(\frac{(y+8)^6}{y^4(y+4)}\right)- \frac{1}{15}m_6\left(\frac{(y + 2)^2(y - 4)^2(y + 8)^2}{y^2(y + 4)^2}\right) + \frac{4}{15}m_6 \left(\frac{(y+2)^6}{y(y+4)}\right),} \label{eq:r24} \\[1.1ex]
	m_{28}(y+8y^{-1}+5) &= \frac{1}{5}m_{14}\left((y+4)^4y^{-2}\right) + \frac{2}{5}m_{14}\left((y+2)^3y^{-1}\right), \label{eq:r28} \\
	m_{30}(y+9y^{-1}+6) &= \frac{1}{10}m_{10}\left((y^2+9y+27)^2y^{-3}\right) + \frac{3}{10}m_{10}\left((y^2+3y+3)^2y^{-1}\right). \label{eq:r30}
	\end{align}
\end{theorem}
\begin{proof} As an illustration, we only prove the first relation. The others can be verified in a similar fashion.
	Let $h_1=\frac{1^4}{2^4}$ and $h_2=\frac{2^4}{4^4}$, then $h_1^6=y^{-1}z^{3}$ and $h_2^6 = y^{1}z^3$.
	We have that \begin{align*}
	c_8(\tau)&=y^2,\\
	c_4(\tau)&=(h_1^{3}+64h_1^{-3})^2 = y^{-1}z^3 + 64^2 yz^{-3} +128, \\
	c_4(2\tau)&=(h_2^{3}+64h_2^{-3})^2 = y^{1}z^3 + 64^2 y^{-1}z^{-3} +128. 
	\end{align*}
	Now the relation follows from the first equation of Lemma \ref{L:modeq} and the first equation of Corollary \ref{C:relation}.
	
In the following table we indicate which exotic relation follows from which relations in Lemma \ref{L:modeq} and Corollary \ref{C:relation}.
\renewcommand{\arraystretch}{1.20205} 
\begin{table}[h!]
	\begin{tabular}{|c|c|c|c|c|c|c|c|c|c|} \hline
	\eqref{eq:q8} &  \eqref{eq:q12} &  \eqref{eq:q16} &\eqref{eq:q18} &\eqref{eq:q12a} &\eqref{eq:q12b} &\eqref{eq:q20} &\eqref{eq:q24} & \eqref{eq:q28} &\eqref{eq:q30} \\	\hline
	\eqref{eq:p8} &  n/a & \eqref{eq:p16} &\eqref{eq:p18} &\eqref{eq:p12a} &\eqref{eq:p12b} &\eqref{eq:p20} &\eqref{eq:p24} & \eqref{eq:p28} &\eqref{eq:p30} \\	\hline
	\eqref{eq:r8} &  \eqref{eq:r12} &  \eqref{eq:r16} &\eqref{eq:r18} &\eqref{eq:r12a} &\eqref{eq:r12b} &\eqref{eq:r20} &\eqref{eq:r24} & \eqref{eq:r28} &\eqref{eq:r30} \\	\hline
	\end{tabular}
\end{table}
\end{proof}

\begin{remark} In particular, Roger's relation \eqref{eq:relP} is a consequence of \eqref{eq:r12b} and the trivial relation (see Proposition \ref{P:trivial}).
But our relation for $V_{24}$ and $V_6$ is different from Roger's \eqref{eq:relQ} because he used Bertin's (different) modular parametrization of the family $V_{24}$.
The relation \eqref{eq:relQ} can be easily obtained from Bertin's parametrization by the same method.
\end{remark}

\begin{example} Let us examine the first relation when $z=-32$ and $z=16$.
When $z=-32$, the relation reads $m_8(-8) = \frac{1}{5}m_4(0) + \frac{2}{5} m_4(648)$ which agrees with our table.
But when $z=16$, the relation reads $m_8(64) = \frac{1}{5}m_4(648) + \frac{2}{5} m_4(648)$.
This does not contradict our table because we are unable to choose $\tau$ lying in both fundamental domains containing $i\infty$.
For our choice in the table, $\tau=\frac{i}{2}$ lies in the fundamental domain containing $i\infty$ for $\Gamma_0^+(4)$ but not for $\Gamma_0^+(2)$.
\end{example}

\section{Appendix for index $>1$}
For completeness, we will treat those families corresponding to Fano 3-folds of index $>1$.
The statement for their Mahler measures follows easily from the trivial relations below.

\renewcommand{\arraystretch}{1.5} 
\begin{table}[h]
	\caption{\label{tab:Lp2} List of Laurent polynomials and Picard-Fuchs equations for $d>1$}
	$\begin{array}{|c|c|c|c|c|}  \hline
	\text{label} & N & f & -(\alpha_3,\alpha_2,\alpha_1,\alpha_0;\beta_1,\beta_0) & e(\tau)\\ \hline
	\mb{P}^3  & 8 & x_1 + x_2 + x_3 + (x_1x_2x_3)^{-1}  & (0,0,0,256;0,0) & 5(0,0,-1,1)\\	\hline
	Q	& 9 & x_1 + x_2 + x_3 + (x_1x_2)^{-1} + (x_1x_3)^{-1} & (0,0,108,0;0,0)  & \frac{10}{3}(0,-1,1)  \\ \hline
	B_1	& 2 & (x_1x_2+x_2x_3+x_3x_1+1)^3/(x_1x_2x_3) & (1728,0,0,0;240,0) &  E_4^{1/2}(2\tau)\\ \hline
	B_2	& 4 & (x_1x_2+x_2x_3+x_3x_1+1)^2/(x_1x_2x_3) & (0,256,0,-64;0,0) & 20(0,-1,1)\\ \hline
	B_3	& 6 & (x_1^{-1} + x_2^{-1})(1+x_3^{-1})(x_3 + x_1x_2 + 1) & (0,108,0,-12;0,0)  & 6(0,1,0,-1) \\ \hline
	B_4 & 8 & (x_1 + x_1^{-1})(x_2 + x_2^{-1})(x_3 + x_3^{-1})  & (0,64,0,0;0,0) & 4(0,-1,0,1)\\ \hline
	B_5	& 10 & x_1 + x_2 + x_3 + x_1^{-1} + x_2^{-1} + x_3^{-1} + (x_1x_2x_3)^{-1} & (0,11,0,16;0,1)  &\frac{5}{2}(0,-1,0,1) \\ \hline \hline
	B_{6a}& 12&x_1 + x_2 + x_3 + x_1^{-1} + x_2^{-1} + (x_1x_2x_3)^{-1} & (0,28,0,128;0,4)  & \frac{1}{2}(0,-1,0,-1,1,1)\\ \hline
	B_{6b}& 12&x_1 + x_2 + x_3 + x_1^{-1} + x_2^{-1} + x_3^{-1}  &  (0,40,0,144;0,8) &2(0,-1,0,1,-1,1) \\ \hline
	\end{array}$
\end{table}



\begin{proposition}\label{P:trivial} We have the following trivial relations 
	\begin{align*} 
    \mtd(\tau,B_i) & = \frac{1}{2} \mtd(2\tau, V_{2i}), \\
	\mtd(\tau,Q) &= \frac{1}{3} \mtd(3\tau, V_{6}), \\
	\mtd(\tau,\mb{P}^3) &= \frac{1}{4} \mtd(4\tau, V_{4}). 
	\end{align*}
Here, our convention is that $2(6a)=12a$ and $2(6b)=12b$.
\end{proposition}
\begin{proof} This is rather clear from Theorem \ref{T:mmmLG} if we compare the columns $e(\tau)$ of Table \ref{tab:Lp} and Table \ref{tab:Lp2}.
\end{proof}

As a sample, we give the corresponding table for $B_{6b}$ because it was intensively studied in the literature \cite{Be1,BFFLM,Ro,BN}.
One can compare it with the table for $V_{12b}$. Note also that $m(c) = m(-c)$ for $B_{6b}$.

\begin{minipage}{15cm}
\paragraph{$$\bf B_{6b}$$}  
\renewcommand{\arraystretch}{1.2}
$$\begin{array}{|c|c|c|c|} \hline
\tau & c(\tau) & \Re(\mtd(\tau)) & \text{remark}\\ \hline
(12,6,1)	& 0  & 24l_{4,-3} & \text{\cite{Be1}} \\	\hline 
(24,8,1)	& 2  & 32\Lt(g_{8},3) & \text{\cite{Be1}}  \\	\hline 
(24,9,1)	& 3  & 30\Lt(g_{15a},3) & \text{\cite{Be1}} \\	\hline 
(24,0,1)	& 6  & 48\Lt(g_{24a},3) &\text{\cite{BFFLM}} \\	\hline 
(24,0,3)	& 10  & 24\Lt(g_{8},3) + 24l_{24,-3} & \text{\cite{BFFLM}}\\	\hline 
(24,0,5)	& 18  & 24\Lt(g_{120b},3) + 12l_{160,-3} & \text{\cite{BFFLM}}\\	\hline 
(24,0,7)	& 30  & 24\Lt(g_{168b},3) + 24l_{21,-8} &\\	\hline 
(24,0,13)  & 102  & 24\Lt(g_{312b},3) + 24l_{13,-24} &\\	\hline 
(24,0,17)  & 198  & 24\Lt(g_{408b},3) + 24l_{136,-3} &\\	\hline 
\end{array}$$
\end{minipage}
\vspace{1cm}

\subsection*{Acknowledgements} The author thanks Hajli Mounir for some discussion related to Lemma \ref{L:sumpartial}.
He also thanks to the computer software Matlab \cite{Mat} and the online database LMFDB \cite{LMFDB},
with which he did lots of numerical experiments in the early stage of this project.

\bibliographystyle{amsplain}

\begin{thebibliography}{99}
\bibitem{ACGK} M. Akhtar, T. Coates, S. Galkin, A. Kasprzyk, \textit{Minkowski polynomials and mutations}, SIGMA Symmetry Integrability Geom. Methods Appl. 8 (2012), Paper 094, 17 pp.
\bibitem{AOP} S. Ahlgren, K. Ono, D. Penniston, \textit{Zeta functions of an infinite family of K3 Surface}, Amer. J. Math. 124 (2002), no. 2, 353--368
\bibitem{Ba} V. Batyrev, \textit{Dual polyhedra and mirror symmetry for Calabi-Yau hypersurfaces in toric varieties,} J. Algebraic Geom. 3 (1994), 493--535. 
\bibitem{Be0} M.-J. Bertin, \textit{Mesure de Mahler d'une famille de polyn\^{o}mes,} J. Reine Angew. Math. 569 (2004), 175--188.
\bibitem{Be1} M.-J. Bertin, \textit{Mesure de Mahler d'hypersurfaces K3}, J. Number Theory 128, 2890--2913 (2008).
\bibitem{Be2} M.-J. Bertin, \textit{A Mahler measure of a K3 surface expressed as a Dirichlet L-series}, Canad. Math. Bull. 55 (2012), no. 1, 26--37.
\bibitem{BFFLM} M.-J. Bertin, A. Feaver, J. Fuselier, M. Lal\'{i}n, M. Manes, \textit{Mahler measure of some singular K3-surfaces}, Research directions in number theory, 149-–169, Contemp. Math., 606, Centre Rech. Math. Proc., Amer. Math. Soc., Providence, RI, 2013.
\bibitem{BP} F. Beukers and C. Peters, \textit{A family of K3 surfaces and $\zeta(3)$,} J. Reine Angew. Math. 351 (1984), 42--54. 
\bibitem{Bo} D. W. Boyd, \textit{Mahler’s measure and special values of L-functions}, Exper. Math. 7, 37--82 (1998).
\bibitem{BN} F. Brunault, M. Neururer, \textit{Mahler measures of elliptic modular surfaces,} Trans. Amer. Math. Soc. 372 (2019), no. 1, 119--152.
\bibitem{BZ} F. Brunault, W. Zudilin, \textit{Many variations of mahler measures - a lasting symphony}, Cambridge University Press, 2020.
\bibitem{CCGK} T. Coates, A. Corti, S. Galkin, A. Kasprzyk, \textit{Quantum periods for 3-dimensional Fano manifolds}, Geom. Topol. 20 (2016), no. 1, 103--256.
\bibitem{Co} D. A. Cox, \textit{Primes of the form $x^2 + ny^2$}, Fermat, class field theory, and complex multiplication, Second edition. Pure and Applied Mathematics, John Wiley \& Sons, Inc., Hoboken, NJ, 2013. xviii+356 pp.
\bibitem{CN} J. Conway, S. Norton, \textit{Monstrous moonshine,} Bull. London Math. Soc. 11 (1979), pp. 308--339.
\bibitem{Dor} C. Doran, \textit{Picard-Fuchs uniformization: Modularity of the mirror maps and mirror-moonshine}, in The arithmetic and Geometry of Algebraic Cycles, Banff 1998, CRM Proc. \& Lecture Notes, Vol. 24 (2000), pp. 257--281.
\bibitem{ES} N. D. Elkies, M. Sch\"{u}tt, \textit{Modular forms and K3 surfaces,}  Adv. Math. 240 (2013), 106--131.
\bibitem{Ga} S. Galkin, \textit{Fano--Mathieu correspondence}, arXiv1809.02738.
\bibitem{Go} V. V. Golyshev, \textit{Classification problems and mirror duality}, Young, Nicholas (ed.), Surveys in geometry and number theory, 88--121, London Math. Soc. Lecture Note Ser., 338, Cambridge Univ. Press, 2007.
\bibitem{GZ} V. V. Golyshev, D. Zagier, \textit{Proof of the gamma conjecture for Fano 3-folds with a Picard lattice of rank one}, (Russian) Izv. Ross. Akad. Nauk Ser. Mat. 80 (2016), no. 1, 27--54; translation in Izv. Math. 80 (2016), no. 1, 24--49.
\bibitem{GKZ} B. Gross, W. Kohnen, D. Zagier, \textit{Heegner points and derivatives of L-series. II}, Math. Ann. 278 (1987) 497--562.
\bibitem{Iw} H. Iwaniec, \textit{Topics in Classical Automorphic Forms}, Graduate Studies in Mathematics, Volume 17, American Mathematical Society, Providence, 1991.
\bibitem {LY} B. H. Lian, S.-T.Yau, \textit{Arithmetic properties of mirror map and quantum coupling,} Comm. Math. Phys. 176 (1996), no. 1, 163--191.
\bibitem{Li} R. Livn\'{e}, \textit{Motivic orthogonal two-dimensional representations of $Gal(\br{\mb{Q}}/\mb{Q})$}, Israel J. Math. 92, 149--156 (1995).
\bibitem{LMFDB} The LMFDB collaboration, The $L$-functions and Modular Forms Database, http://www.lmfdb.org.
\bibitem{Mat} The Math Works, Inc., MATLAB, version 2018b, https://www.mathworks.com.
\bibitem{Pr} V. Przyjalkowski, \textit{Weak Landau-Ginzburg models for smooth Fano threefolds}, Izvestiya: Mathematics. 77, no.4 (2013).
\bibitem{Ro} M. D. Rogers, \textit{New $_5F_4$ hypergeometric transformations, three-variable Mahler measures, and formulas
for $1/\pi$}, Ramanujan J. 18 (2009), 327--340.
\bibitem{Sa0} D. Samart, \textit{Mahler measures of hypergeometric families of Calabi-Yau varieties}, Thesis, 2014.
\bibitem{Sa} D. Samart, \textit{Three-variable Mahler measures and special values of modular and Dirichlet L-series}, Ramanujan J. 32 (2013), no. 2, 245--268.
\bibitem{Sc} M. Sch\"{u}tt, \textit{CM newforms with rational coefficients,} Ramanujan J. 19 (2009), no. 2, 187–-205.
\bibitem{Sh} J. Shurman, \textit{Hecke characters, classically and idelically}, http://people.reed.edu/\~{}jerry/361/lectures/heckechar.pdf.
\bibitem{VY} H. Verrill, N. Yui, \textit{Thompson series, and the mirror maps of pencils of K3 surfaces,} The Arithmetic and Geometry of Algebraic Cycles, CRM Proc. Lecture Notes, vol. 24, pp. 399--432, Amer. Math. Soc., Providence, RI (2000).
\bibitem{Vi} F. R. Villegas, \textit{Modular Mahler measures}, I. Topics in number theory (University Park, PA, 1997), 17--48.
\bibitem{Yi} Y.-F. Yang \textit{On differential equations satisfied by modular forms,} Math. Z. 246 (2004), no. 1-2, 1--19.
\bibitem{Za} D. Zagier, \textit{Traces of singular moduli}, Motives, polylogarithms and Hodge theory, Part I (Irvine, CA, 1998), 211--244, Int. Press Lect. Ser., 3, I, Int. Press, Somerville, MA, 2002.
\bibitem{Zabook} D. Zagier, \textit{Elliptic modular forms and their applications,} In The $\mathit{1{-}2{-}3}$ of Modular Forms: Lectures at a Summer School in Nordfjordeid, Norway, K. Ranestad (ed.), Universitext, Springer-Verlag, Berlin (2007), 1--103.


\end{thebibliography}

\end{document}